\documentclass{amsart}
\usepackage{amsthm}
\newtheorem{theorem}{Theorem}
\newtheorem{definition}[theorem]{Definition}

\newtheorem{lemma}[theorem]{Lemma}
\newtheorem{proposition}[theorem]{Proposition}
\newtheorem{corollary}[theorem]{Corollary}
\newtheorem{example}[theorem]{Example}
\newtheorem{remark}[theorem]{Remark}

\newcommand{\E}{\mathbb{E}}
\newcommand{\R}{\mathbb{R}}
\newcommand{\e}{\text{e}}

\allowdisplaybreaks

\title{Cointegration in continuous time for factor models}
\author{Fred Espen Benth and Andre S\"uss}
\address{Fred Espen Benth \\
Department of Mathematics \\
University of Oslo\\
P.O. Box 1053, Blindern\\
N--0316 Oslo, Norway}
\email[]{fredb\@@math.uio.no}
\address{Andre S\"uss\\
Centre for Advanced Study \\
Drammensveien 78, N-0271 Oslo, Norway}
\email[]{suess.andre\@@web.de}
\date{\today}
\thanks{F. E. Benth acknowledges financial support from the project FINEWSTOCH funded by the Norwegian Research Council. A. S\"uss acknowledges partial financial support by the grant MTM 2015-65092-P, Secretaria de estado de investigacion, desarrollo e innovacion, Ministerio de Economia y Competitividad, Spain.}

\begin{document}

\begin{abstract}
We develop cointegration for multivariate continuous-time stochastic processes, both in finite and infinite dimension. Our 
definition and analysis are based on factor processes and operators mapping to the space of prices and cointegration. 
The focus is on commodity markets, where both spot and forward prices are analysed in the context of
cointegration. We provide many examples which include the most used continuous-time pricing models, including
forward curve models in the Heath-Jarrow-Morton paradigm in Hilbert space.    
\end{abstract}

\maketitle

\section{Introduction}
\label{sec:intro}

We aim at developing a formalism to the concept of {\it cointegration} in continuous time. Cointegration has since the seminal paper 
of Engle and Granger~\cite{EG} become a very popular concept for stochastic modelling of dependent time series of data, in particular in
economics. For example, the price series of two financial assets can be non-stationary, while one may find that a linear combination of these is stationary. Cointegration provides a framework
for analysing and modelling time series that explains such observable features in data. 

Although there has been a huge development in continuous-time financial models over the last decades, the literature on
cointegration for continuous time stochastic processes and its application to finance is relatively scarce. A non-exhaustive list of
papers in this stream of research include Comte~\cite{comte}, Duan and Pliska~\cite{DP}, Duan and Theriault~\cite{DT}, Nakajima and Ohashi~\cite{NO}, Paschke and Prokopczuk~\cite{PP},
Benth and Koekebakker~\cite{BK}, and recently Farkas et al.~\cite{Farkas}. In the present paper we formalise ideas 
on cointegration in continuous time for factor processes, and extend these to cointegration for stochastic processes with infinite
dimensional state space. The latter will provide a theoretical framework for studying cointegration in forward and futures markets,
say.

Comte~\cite{comte} presents an in-depth analysis on classical cointegration and its extension to continuous-time 
models, where continuous-time autoregressive moving average processes (CARMA) play a central role. Duan and Pliska~\cite{DP} 
analyse a specific cointegrated asset price model, and show that pricing options will not be influenced by cointegration. Their paper
has triggered many theoretical and empirical studies, including Nakajima and Ohashi~\cite{NO}, Paschke and Prokopczuk~\cite{PP},
Benth and Koekebakker~\cite{BK} and Farkas et al.~\cite{Farkas}. Duan and Theriault~\cite{DT} extend cointegration to
continuous-time forward price models. Benth and Koekebakker~\cite{BK}, and more recently Benth~\cite{benth-eberleinfest} 
focus on the relationship between cointegration in spot and forward markets, and propose cointegration models for forward markets. 
Contrary to the conclusions of Duan and Pliska~\cite{DP}, these two papers argue that in commodity markets the pricing measure 
may preserve cointegration. We refer to Back and Prokopczuk~\cite{BP} for a review of modelling and pricing in commodity markets.

Starting with spot price models, we discuss a framework for cointegration based on factor models. Our concept makes use of a set of stochastic processes, which we call factors, which explains the dynamics of prices via a linear transformation. This yields a vector-valued
price dynamics, for which one can introduce the concept of cointegration. The following example is frequently referred to in the text, and explains our ideas in a simple setting. 
\begin{example}
\label{motivating-ex}
Consider the classical spot price model for two 
commodity markets, given 
by the two-factor model;
\begin{equation}
\label{def:canonical}
S_i(t)=X_i(t)+X_3(t), i=1,2.
\end{equation} 
We assume $X_3(t)=\mu t+\sigma B_3(t)$ being a drifted Brownian motion and $X_i(t)$ being two Ornstein-Uhlenbeck processes,
$$
dX_i(t)=-\alpha_i Y_i(t)\,,dt+\eta_i\,dB_i(t),
$$ 
with constants $\alpha_i>0, \eta_i>0, i=1,2$. Here, $(B_1,B_2,B_3)$ is a trivariate
Brownian motion, possibly correlated.
This model was proposed in the univariate case by Lucia and Schwartz~\cite{LS} for
electricity spot prices and extended to cross-commodity markets by Paschke and Prokopczuk~\cite{PP-coint} for oil markets
(see also Duan and Pliska~\cite{DP} and Benth and Koekebakker~\cite{BK} for general analysis). For example, 
$(S_1,S_2)$ can model the joint spot price dynamics in the coal and electricity market, or in two different 
electricity markets. Since the bivariate Ornstein-Uhlenbeck process 
$(X_1,X_2)$ admits a limiting Gaussian distribution, the price difference process $S_1(t)-S_2(t)$ will have a limiting distribution.
On the other hand, each marginal price process $S_i$ is non-stationary since the drifted Brownian motion $X_3$ is (unless $\mu=\sigma=0$). According
to Duan and Pliska \cite{DP}, the processes $S_1$ and $S_2$ are {\it cointegrated}. We notice that the definition of the bivariate price
process $(S_1,S_2)$ involves three factor processes  $X_1$, $X_2$ and $X_3$, and a linear combination of these. Introducing the matrix
\begin{equation}
\label{eq:P-LS}
\mathcal P=\left[\begin{array}{ccc} 1 & 0 & 1 \\ 0 & 1 & 1 \end{array}\right],
\end{equation}
we represent the vector $\mathbf S:=(S_1,S_2)^{\top}$ as $\mathbf S(t)=\mathcal P (X_1(t),X_2(t),X_3(t))^{\top}$. 
We assume that all
elements $\mathbf x\in\R^n$, $n\in\mathbb N$, are coloumn vectors, and $\mathbf x^{\top}$ is the transpose of 
$\mathbf x$. Cointegration
is achieved since there exists a vector $\mathbf c=(1,-1)^{\top}\in\R^2$ such that the process 
$\mathbf c^{\top}\mathbf S = X_1 - X_2$ admits a limiting distribution.  
\end{example}

Based on multivariate spot price models of the form introduced in this example, we analyse forward prices derived from processes
with certain affinity properties. In this context, polynomial processes (see Filipovic and Larsson~\cite{FL}) constitute an important case, 
along with the more specific CARMA processes. We present results on the cointegration relationship between spot and forward markets, 
with a particular attention to pricing measures and the application to commodity markets. 

Our analysis of spot and forward markets motivates the definition of cointegration for stochastic processes in infinite dimensions. 
We introduce a concept for modelling cointegrated forward curves, following the HJM-paradigm 
(see Heath, Jarrow and Morton~\cite{HJM}) of modelling forward prices directly rather than explaining these via spot models
(we refer to Benth, \v{S}altyt\.e Benth and Koekebakker~\cite{BSBK} for an extensive analysis of forward modelling in energy markets). 
Since forward curves can be modelled as stochastic processes in Hilbert space of real-valued functions on $\R_+$ 
(see Benth and Kr\"uhner \cite{BK-CIMS,BK-sifin}),
we concentrate our analysis on formulating cointegration via linear operators on Hilbert spaces. We show how cointegration in Hilbert space can be related to the finite dimensional case. It turns out that product Hilbert spaces provide a natural framework for
modelling, and we give several examples including infinite dimensional factor processes capturing stationary and non-stationary effects as well as non-Gaussianity.  We also include a discussion of some 
recent empirical studies on forward gas markets by Geman and Liu~\cite{GemanLiu} viewed in our cointegration context. 

The results of this paper is presented as follows: in Section 2 we define and analyse cointegration for multivariate spot price 
models based on factor processes. The question of forward pricing in cointegrated spot markets is analysed in Section 3, where we
give a description of cointegration of forwards. Finally, in Section 4 we introduce cointegration for Hilbert-space valued stochastic processes, and apply this to cross-commodity forward prices modelled within the HJM-approach.


\section{Cointegration for factor models}

Suppose that $(\Omega,\mathcal F,\mathbb P)$ is a complete probability space equipped with
a right-continuous filtration $\{\mathcal F_t\}_{t\geq 0}$ where $\mathcal F_t$ contains all sets of $\mathcal F$ of probability zero
(i.e., satisfying the {\it usual conditions}).
Let $\{\mathbf S(t)\}_{t\geq 0}\in\R^d$ be $d$ asset prices in a given market,
where we define
\begin{equation}
\label{eq:basic-def}
\mathbf S(t)=\mathcal P\mathbf X(t)\,, t\geq 0
\end{equation}
for an adapted stochastic process $\{\mathbf X(t)\}_{t\geq 0}\in\R^n$ and 
$\mathcal P\in\R^{d\times n}$. The matrix $\mathcal P$ is hereafter referred to as the 
{\it pricing matrix}, and $\mathbf X$ the {\it factor process} of the market. 
For convenience, we assume that the number of factors $n$ is 
at least equal to the number of assets $d$, i.e., $n\geq d$. We reserve the notation
$\{\mathbf e_i\}_{i=1}^k$ for the canonical basis vectors in $\R^k$, where the dimension $k\in\mathbb N$ will be clear from the 
context.
\begin{definition}
The pricing matrix $\mathcal P$ is called {\it minimal} if all factors $X_i(t), i=1,\ldots,n$ are
represented in $\mathbf S(t)$. 
\end{definition}
The definition simply says that we have all necessary factors to define the price dynamics $\mathbf S$. It {\it does not} say that all factors are present in each
price coordinate $S_j(t), j=1,\ldots,d$. 
\begin{lemma}
$\mathcal P$ is minimal if and only if $\mathbf e_i\notin\text{ker}(\mathcal P)$ for all $i=1,\ldots,n$.
\end{lemma}
\begin{proof}
Obviously,
$$
\mathbf X(t)=\sum_{j=1}^nX_j(t)\mathbf e_j,
$$
where $X_j(t)=\mathbf X^{\top}(t)\mathbf e_j$. Hence,
$$
\mathbf S(t)=\sum_{j=1}^nX_j(t)\mathcal P \mathbf e_j.
$$ 
If $\mathbf e_i\in\text{ker}(\mathcal P)$ for given $i\in\{1,\dots,n\}$, then $\mathbf S$ will not depend on
$X_i$. Opposite, if $\mathbf S$ is not depending on $X_i$, then $\mathcal P \mathbf e_i=0$. 
\end{proof}
We restrict our considerations to minimal pricing matrices $\mathcal P$. 
We also confine ourselves to non-degenerate markets, that is,
markets where we have $d$ distinct price processes for the $d$ assets. Hence, we assume that $\mathbb P(S_k(t)=S_j(t)\,, t\geq 0)=0$ for any $k,j=1,\ldots,d$ with $k\neq j$. A sufficient condition for this to hold is when the $d$ row vectors of $\mathcal P$ are linearily 
independent, that is, when $\text{rank}(\mathcal P)=d$. We restrict our analysis to this case. 

Assuming the price process $\mathbf S$ is defined by \eqref{eq:basic-def} with $\mathcal P$ being a minimal pricing matrix having
full rank, we define
cointegration as follows:
Let $P_{\mathbf X}(t,\cdot)$ denote the probability distribution of $\mathbf X(t)$ defined on $\mathcal{B}(\R^n)$, the 
Borel $\sigma$-algebra on $\R^n$. 
For any $\mathbf x\in\R^n$, we denote by $P_{\mathbf x^{\top}\mathbf X}(t,\cdot)$ the probability
distribution of the real-valued random variable $\mathbf x^{\top}\mathbf X(t)$. Furthermore, denote by 
$\Psi_{\mathbf X}(t,\mathbf z)$ the characteristic function of $\mathbf X(t)$,
defined for $\mathbf z\in\R^n$ as
$$
\Psi_{\mathbf X}(t,\mathbf z)=\E\left[\e^{\mathrm{i}\mathbf z^{\top}\mathbf X(t)}\right].
$$
Sometimes, one is using the cumulant function instead of the characteristic function, where
the cumulant $\kappa_{\mathbf X}(t,\mathbf z):=\log\Psi_{\mathbf X}(t,\mathbf z)$ with $\log$ denoting the distinguished logarithm
(see  e.g. Sato~\cite{Sato}). We see that 
$\Psi_{\mathbf X}(t,\mathbf z)=\widehat{P}_{\mathbf X}(t,\mathbf z)$, where $\widehat{P}_{\mathbf X}(t,\mathbf z)$ is the Fourier transform of the 
distribution $P_{\mathbf X}(t,\cdot)$.  
\begin{definition}
[Definition of cointegration] We say that $\mathbf S$ is cointegrated if there exists
$\mathbf c\in\R^d$ and a probability distribution $\mu_{\mathbf c}$ on $\mathcal B(\R)$ such that 
$P_{\mathbf c^{\top}\mathcal P \mathbf X}(t,\cdot)$ converges to $\mu_{\mathbf c}$ when $t\rightarrow\infty$. We call
$\mathbf c$ a cointegration vector for $\mathbf S$.
\label{def:cointegration}
\end{definition}
In the definition of cointegration, the convergence is in the sense of probability measures (see Def.~2.2 in Sato~\cite{Sato}).  
The definition of cointegration means that there exists a linear combination of the 
price process vector $\{\mathbf S(t)\}_{t\geq 0}$ which admits a limiting probability distribution, where the linear combination is 
represented by the cointegration vector $\mathbf c$. We recall that $\mathbf c^{\top}\mathbf S(t)=\mathbf c^{\top}\mathcal P \mathbf X(t)$, and thus
$P_{\mathbf c^{\top}\mathcal P \mathbf X}(t,\cdot)$ is the probability distribution
of $\mathbf c^{\top}\mathbf S(t)$, which must converge to a probability distribution when time tends to infinity in order to achieve cointegration.

Definition~\ref{def:cointegration} includes trivially the case $\mathbf c=\mathbf 0\in\R^{d}$, since 
$P_{\mathbf 0^{\top}\mathcal P \mathbf X}(t,\cdot)=\delta_{\mathbf 0}(\cdot)$
with $\delta_{\mathbf 0}$ being the Dirac measure at zero. From a practical viewpoint, we are obviously not interested in this 
degenerate case of
a cointegration vector, but include $\mathbf c=\mathbf 0$ in any case for completeness. 
If we choose $n=d$ and $\mathcal P=I$, the $d\times d$ identity matrix, we have 
$\mathbf S=\mathbf X$. Thus the definition of cointegration can also be directly applied for the 
factor process.

As the next result shows, cointegration can be characterized by convergence of characteristic functions
when time tends to infinity:
\begin{proposition}
[Cumulant characterisation of cointegration] If $S$ is cointegrated with cointegration vector 
$\mathbf c\in\R^d$, 
then $\lim_{t\rightarrow\infty}\Psi_{\mathbf X}(t,z\mathcal P^{\top}\mathbf c)=\Psi_{\mu_{\mathbf c}}(z)$ uniformly (in $z\in\R$) on any compact set, with $\Psi_{\mu_{\mathbf c}}$ being the characteristic function of the distribution $\mu_{\mathbf c}$. 
Opposite, if there exists a $\mathbf c\in\R^d$ and a complex-valued function $z\mapsto\Psi_{\mathbf c}(z)$ on $\R$ which is continuous at $z=0$ 
such that $\lim_{t\rightarrow\infty}\Psi_{\mathbf X}(t,z\mathcal P^{\top}\mathbf c)=\Psi_{\mathbf c}(z)$
for every $z\in\R$, then
$\mathbf S$ is cointegrated with cointegration vector $\mathbf c$.
\label{prop:characterisation}
\end{proposition}
\begin{proof}
By Definition~\ref{def:cointegration} of cointegration we have that $P_{\mathbf c^{\top}\mathcal P \mathbf X}(t,\cdot)\rightarrow\mu_{\mathbf c}$
for a probability distribution $\mu_{\mathbf c}$. It is well-known (see e.g., Sato~\cite[Prop.~2.5 (vi)]{Sato}) that this
implies $\Psi_{\mathbf X}(t,z\mathcal P^{\top}\mathbf c)\rightarrow\Psi_{\mu_{\mathbf c}}(z)$ as $t\rightarrow\infty$ on any compact set. This shows the first part. 

If for a given $\mathbf c\neq 0\in\R^d$, there exists a $\Psi_{\mathbf c}$ such that $\Psi_{\mathbf X}(t,z\mathcal P^{\top}\mathbf c)\rightarrow\Psi_{\mathbf c}(z)$ for every $z\in\R$, then Sato \cite[Prop.~2.5 (viii)]{Sato} ensures the 
existence of a probability distribution $\mu_{\mathbf c}$ with characteristic function $\Psi_{\mathbf c}$ as long as
$z\mapsto\Psi_{\mathbf c}(z)$ is continuous at $z=0$. But this means that $P_{\mathbf c^{\top}\mathcal P \mathbf X}(t,\cdot)\rightarrow\mu_{\mathbf c}$.
The result follows. 
\end{proof}

If $\mathbf c$ is a cointegration vector for $\mathbf S$, then by appealing to Definition~\ref{def:cointegration}, 
$\mathcal P^{\top}\mathbf c\in\R^d$ is a cointegration vector for the factor process $\mathbf X$. Opposite, if $\mathbf a\in\R^n$ is a cointegration vector for $\mathbf X$, that is, there exists a distribution
$\mu$ such that $P_{\mathbf a^{\top}\mathbf X}(t,\cdot)\rightarrow\mu$, then for any $\mathbf c\in\R^d$ for which $\mathcal P^{\top}\mathbf c=\mathbf a$ becomes a 
cointegration vector for $\mathbf S$. We note that such a $\mathbf c$ may fail to exist, so even if $\mathbf X$ admits a cointegration
vector, it may not give rise to a cointegration vector for $\mathbf S$. If $d=n$, then $\mathbf c:=\mathcal P^{-\top}\mathbf a$ since $\mathcal P$ is invertible due to 
the full rank assumption. However, the typical situation is that $d<n$, and then the linear system $\mathcal P^{\top}\mathbf c=\mathbf a$ is over-determined and in general will not possess a solution. 

We have the following convenient definition:
\begin{definition}
Denote by $\mathcal C_{\mathbf X}$ the set of all cointegration vectors for $\mathbf X$ 
and $\mathcal C_{\mathbf S}$ the set of all cointegration vectors for $\mathbf S$. 
\end{definition}
We note from the discussion above that if $\mathbf c\in\mathcal C_{\mathbf S}$, then $\mathcal P^{\top}\mathbf c\in\mathcal C_{\mathbf X}$. Hence,
$\mathcal P^{\top}\mathcal C_{\mathbf S}\subset\mathcal C_{\mathbf X}$. For many specifications of $\mathcal P$, this inclusion
is strict, telling that the set of cointegration vectors for $\mathbf S$ is restricted compared to the range of cointegration possibilities given by the vector $\mathbf X$. 
But if $\mathbf a\in\mathcal C_{\mathbf X}$ is in the image of $\mathcal P^{\top}$, then we have the existence of a unique $\mathbf c\in\R^d$ such that $\mathcal P^{\top}\mathbf c=\mathbf a$,  that is, $\mathbf c\in\mathcal C_{\mathbf S}$. Uniqueness of 
$\mathbf c$ follows from the fact that the $n\times d$ matrix $\mathcal P^{\top}$ has
full column rank $d$. In particular, if $\mathcal C_{\mathbf X}\subset\text{Range}(\mathcal P^{\top})$, then 
$\mathcal P^{\top}\mathcal C_{\mathbf S}=\mathcal C_{\mathbf X}$. 

We define a {\it cointegrated pricing system}:
\begin{definition}
If $\mathcal P\in\R^{d\times n}$ is a pricing matrix, i.e., minimal and with $\text{rank}(\mathcal P)=d$, and $\mathbf c\in\R^d$ is such that 
$\mathcal P^{\top}\mathbf c\in\mathcal S_{\mathbf X}$, 
we say that $(\mathcal P,\mathbf c)$ is a cointegrated pricing system for the factor process $\mathbf X$. 
\end{definition}
If $(\mathcal P,\mathbf c)$ is a cointegrated pricing system for the factor process $\mathbf X$, we can define
a system of prices $\mathbf S(t)=\mathcal P\mathbf X(t)$ which becomes cointegrated for the vector $\mathbf c$, 
according to Definition~\ref{def:cointegration}. To a given $\mathbf a\in\mathcal S_{\mathbf X}$, there may exist many 
cointegrated pricing systems $(\mathcal P,\mathbf c)$; indeed all possible combinations of 
pricing matrices $\mathcal P$ and vectors $\mathbf c$ such that
$\mathcal P^{\top}\mathbf c=\mathbf a$. 

Let us return to Example~\ref{motivating-ex}, where we considered two spot price dynamics given by 
\eqref{def:canonical}. We recall the pricing matrix $\mathcal P$ in \eqref{eq:P-LS}, and the factor process $\mathbf X(t)=(X_1(t),X_2(t),X_3(t))$
for $X_3$ a drifted Brownian motion and $(X_1,X_2)$ a bivariate Ornstein-Uhlenbeck process.  
As $\mathcal P$ has two independent row
vectors, it has full rank, $\text{rank}(\mathcal P)=2$. Moreover, we easily see that $\mathcal P\mathbf e_i\neq 0$ for $i=1,2,3$. Indeed,
$\text{ker}(\mathcal P)$ has dimension 1 and is spanned by the vector $(1,1,-1)$. We conclude that $\mathcal P$ is a pricing matrix
satisfying the assumptions of minimality and full rank. In this model, $(X_1,X_2)$ is a 2-dimensional
Ornstein-Uhlenbeck process, 
$$
X_i(t)=X_i(0)\e^{-\alpha_i t}+\int_0^t\eta_i\e^{-\alpha_i(t-s)}\,dB_i(s)\,, i=1,2.
$$
We find
$$
(X_1(t),X_2(t))\stackrel{d}{\rightarrow}\mathcal N\left(0,C\right)
$$
where the covariance matrix $C\in\R^{2\times 2}$ is 
$$
C=\left[\begin{array}{cc} \frac{\eta_1^2}{2\alpha_1} & \rho\frac{\eta_1\eta_2}{\alpha_1+\alpha_2} \\ 
\rho\frac{\eta_1\eta_2}{\alpha_1+\alpha_2} & \frac{\eta_2^2}{2\alpha_2}\end{array}\right]\,.
$$
Here, $\rho$ is the correlation between $B_1$ and $B_2$, and $\stackrel{d}{\rightarrow}$ denotes limit in
distribution. Hence, since $(X_1,X_2)$ has a limiting distribution, we find from the non-stationarity of $X_3$ that 
$\mathcal C_{\mathbf X}=\{\mathbf a\in\R^3\,|\,a_3=0\}$. In particular, $\mathcal C_{\mathbf X}$ is a vector space with basis vectors $\mathbf e_1$ and 
$\mathbf e_2$. We remark that in general,
$\mathcal C_{\mathbf X}$ does not need to be a vector space. If we for example substitute $X_1$ and $X_2$ with two stationary
stochastic processes which are not jointly stationary, we have that a linear combination of the two may fail to be stationary even 
though they are marginally stationary. We find further that $\mathcal C_{\mathbf S}$ is the vector space spanned by the vector $(1,-1)^{\top}$.
Finally, the range of $\mathcal P^{\top}$ is spanned by the two vectors $(1,0,1)^{\top}$ and $(0,1,1)^{\top}$, i.e., the row vectors
of $\mathcal P$. Thus, if $\mathbf a\in\mathcal C_{\mathbf X}$, then $\mathbf a$ is in the range of $\mathcal P^{\top}$ {\it only} when $\mathbf a=k(1,-1,0)^{\top}$ for
$k\in\R$. Therefore, $\mathcal P^{\top}\mathcal C_{\mathbf S}\subset\mathcal C_{\mathbf X}$, with a strict inclusion in this case. 
From these considerations, we also see that there exists many pricing systems $(\mathcal P,\mathbf c)$,
indeed, for a fixed $\mathcal P$ we have a continuum of $\mathbf c\in\mathcal C_{\mathbf S}$. But we may also choose different
pricing matrices. For example, if 
$$
\mathcal P=\left[\begin{array}{ccc} a & b & w \\ u & v & 1 \end{array}\right]\,.
$$ 
for any $a,b,u,v,w\in \R$ such that $\mathcal P$ is minimal and non-degenerate, we can use 
$\mathbf c=(1,-w)^{\top}$ to define a pricing system, where $\mathbf c^{\top}\mathcal P \mathbf X(t)=(a-uw)X_1(t)+(b-vw)X_2(t)$. Such a
pricing matrix $\mathcal P$ is relevant when modelling two commodities that do not share the same denominator. For example,
gas and coal typically have different energy units than power, and we will have a conversion factor (heat rate) between them modelled by
$w$ in the present context.

The particular example discussed above motivates some further analysis of the set $\mathcal C_{\mathbf X}$. In many situations, as in the example, we can single out a subset of factors from $\mathbf X$ which has a limit in distribution, i.e.,
$\mathbf X^m(t):=(X_1(t),\ldots,X_m(t))^{\top}$ with $m\leq n$ for which $P_{\mathbf X^m}(t,\cdot)\rightarrow\mu^m$ for
a probability distribution $\mu^m$ on $\R^m$ as $t\rightarrow\infty$. Then $\mathcal C_{\mathbf X}^m\subset\mathcal C_{\mathbf X}$, where
$$
\mathcal C_{\mathbf X}^m:=\{\mathbf a\in\R^n\,|\,a_{m+1}=...=a_n=0\}.
$$
Remark that we do not in general have equality between $\mathcal C_{\mathbf X}^m$ and $\mathcal C_{\mathbf X}$ as there may be 
cointegration between some of the factors $X_{m+1},\ldots,X_n$ that may not hold jointly with the first $m$ factors. 
For convenience, we have assumed that the subset of factors which has a limiting distribution consists of the first $m$. Since we may re-label the factors, this assumption is of course without loss of generality. We observe that $\mathcal C_{\mathbf X}^m$ is a vector space, and that in
the case $m=n$, we trivially have $\mathcal C_{\mathbf X}^m=\mathcal C_{\mathbf X}=\R^n$. When $m<n$, any 
$(\mathcal P,\mathbf c)$ such that 
$\mathcal P^{\top}\mathbf c\in\mathcal C_{\mathbf X}^m$ will be a cointegrated pricing system for $\mathbf X$.
We observe that these considerations are in line with the example above, where $\mathcal C_{\mathbf X}^2=\mathcal C_{\mathbf X}$ since the two
first factors have jointly a limiting distribution, while the last factor is non-stationary. We have the following general result: 
\begin{lemma}
Suppose $P_{\mathbf X^{n-1}}(t,\cdot)$ has a limiting distribution,  while 
$P_{X_n}(t,\cdot)$ does not have a limiting distribution. If $X_n$ is independent of $\mathbf X^{n-1}$, then
$\mathcal C_{\mathbf X}=\mathcal C_{\mathbf X}^{n-1}$. 
\end{lemma}
\begin{proof}
Let $\mathbf c\in\mathcal C_{\mathbf X}$ with $c_n\neq 0$. By independence, we find for $z\in\R$
\begin{align*}
\Psi_{\mathbf c^{\top}\mathbf X}(t,z)&=\E[\e^{\mathrm{i}z\mathbf c^{\top}\mathbf X(t)}] \\
&=\E[\e^{\mathrm{i}z(c_1X_1(t)+\cdots+c_{n-1}X_{n-1}(t)}]\E[\e^{\mathrm{i}zc_nX_n(t)}] \\
&=\Psi_{\mathbf X^{n-1}}(t,z(c_1,\ldots,c_{n-1})^{\top})\Psi_{X_n}(t,zc_n)\,.
\end{align*}
For every $z$, $\Psi_{\mathbf X^{n-1}}(t,z(c_1,\ldots,c_{n-1})^{\top})$ will have a limit, while there exists a Borel set
$A_0$ with positive Lebesgue measure such that $\Psi_{X_n}(t,x)$ does not have a limit for every $x\in A_0$
(this could be the whole of the real line, or some subset with infinite Lebesgue measure). But then 
for all $z\in A_0/c_n$ we have that 
 $\Psi_{X_n}(t,zc_n)$ does not have a limit, and in conclusion $\Psi_{\mathbf c^{\top}\mathbf X}(t,z)$ does not have a limit
for every $z\in\R$ as $t\rightarrow\infty$. This violates the assumption that $\mathbf c\in\mathcal C_{\mathbf X}$ with $c_n\neq 0$. Thus, $c_n=0$, showing the claim.
\end{proof}
Remark that in Example~\ref{motivating-ex}, the non-stationary drifted Brownian motion is not necessarily independent 
of the two other factors, showing that the assumption of independence is sufficient, but not necessary.

Notice that if $\mathbf X^m$ admits a stationary limit, and $(X_{m+1},...,X_{n})$ is dependent on $X^m$, we may 
have non-trivial $\mathbf c\in\mathcal C_{\mathbf X}\backslash\mathcal C_{\mathbf X}^m$. Indeed, consider $n=3$ and the processes
$$
X_1(t)=\int_0^t\exp(-\alpha_1(t-s))\,dB_1(s)\,,
$$
and
$$
X_i(t)=\mu t+\int_0^t\exp(-\alpha_i(t-s))\,dB_i(s), i=2,3\,,
$$
for constants $\mu, \alpha_i>0, i=2,3$ and a trivariate Brownian motion $(B_1,B_2,B_3)$ being correlated. 
Then $\mathbf X=(X_1,X_2,X_3)^{\top}$ have dependent coordinates, and for any vector $\mathbf c=(a,b,-b)^{\top}\in\R^3, a,b\in\R$, we find that 
$$
\mathbf c^{\top}\mathbf X(t)=a\int_0^t\e^{-\alpha_1(t-s)}\,dB_1(s)+b\left(\int_0^t\e^{-\alpha_2(t-s)}\,dB_2(s)-\int_0^t\e^{-\alpha_3(t-s)}\,dB_3(s)\right)
$$ 
which will converge in distribution to a normally distributed random variable with zero mean as $t\rightarrow\infty$. 
Hence, $\mathbf c\in\mathcal C_{\mathbf X}$. Here, $X_1$ has a limit in distribution, while
$X_i, i=2,3$ both will have a mean $\mu t$ and thus there does not exist any limiting distribution. This is 
an example with $m=1$ and $n=3$. We remark that the example is slightly pathological, as we could have 
assumed $n=4$ with $X_4(t)=\mu t$, and defined $\widetilde{X}_i(t)=\int_0^t\exp(-\alpha_i(t-s)\,dB_i(s), i=2,3$. Then, with 
$\mathbf X:=(X_1,\widetilde{X}_2,\widetilde{X}_3,X_4)^{\top}$ we are back to the situation with $m=n-1=3$ and $X_4$ being (trivially) independent of $\mathbf X^3=(X_1,\widetilde{X}_2\widetilde{X}_3)^{\top}$.

We have the following remark, which gives a practical consequence of our considerations so far: 
\begin{remark}
In a practical application we can model a system of $d$ commodity price dynamics with cointegration as follows: first, we assume that
we have $m$ factor processes which jointly admit a limiting distribution, and $n-m$ non-stationary processes, with $n\geq d$.  
Then we know that any $(\mathcal P,\mathbf c)$ such that $\mathcal P^{\top}\mathbf c\in\mathcal C^m_{\mathbf X}$ will be a cointegrated pricing system. This
provides us with a constraint on the possible specifications of $(\mathcal P,\mathbf c)$ which can be used in the next step on specifying 
parametric models for the factor processes and estimating on data. As long as we know that $\mathbf X^m$ admits a limiting distribution,
we can characterize a set of admissible pricing systems $(\mathcal P,\mathbf c)$ before any further specification and estimation on data. 
\end{remark}

In the anaysis so far we have exclusively thought of the price dynamics $\mathbf S$ in \eqref{eq:basic-def} as being on an arithmetic form. However, commonly one models cointegration on the logarithm of prices, $\ln\mathbf  S:=(\ln S_1,\ldots,\ln S_d)^{\top}$. If we suppose that 
$\ln \mathbf S$ satisfies 
\begin{equation}
\label{def:geom-factor-model}
\ln\mathbf  S(t)=\mathcal P\mathbf  X(t),\,t\geq 0,
\end{equation}
we can repeat the analysis above for a {\it geometric} price dynamics. Energy markets like gas and power have frequently experienced 
negative prices, and hence an arithmetic price dynamics may be attractive.

\subsection{Particular model specifications}

Recalling Example~\ref{motivating-ex}, we may for the cross-commodity spot price dynamics \eqref{def:canonical} assume a general (non-stationary) dynamics $X_3$ and a bivariate process 
$\mathbf Y:=(X_1,X_2)^{\top}$ with the property that $P_{\mathbf Y}(t,\cdot)\rightarrow P_{\infty}(\cdot)$ for some probability distribution $P_{\infty}$ on $\R^2$. Then, we find for any $\mathbf c=(k,-k)^{\top}\in\mathcal C_{\mathbf S}$ that 
the characteristic function of the random variable $\mathbf c^{\top}\mathcal P\mathbf X(t)$ is
$$
\Psi_{\mathbf c^{\top}\mathcal P\mathbf X}(t,z)=\E\left[\e^{\mathrm{i}zk(X_1(t)-X_2(t))}\right]=\Psi_{\mathbf X}(t,z\mathcal P^{\top}\mathbf c)=
\widehat{P}_{\mathbf Y}(t,z\mathbf c^{\top})\rightarrow\widehat{P}_{\infty}(z\mathbf c^{\top})\,,
$$
for every $z\in\R$ as $t\rightarrow\infty$. But then by Prop.~\ref{prop:characterisation}, $\mathbf c^{\top}\mathcal P\mathbf X(t)$ 
has a limiting distribution
$\mu_{\mathbf c}$ with characteristic function $\widehat{\mu}_{\mathbf c}(z)=\widehat{P}_{\infty}(z\mathbf c^{\top})$. This shows that we
may significantly go beyond the dynamics discussed in \eqref{def:canonical} that preserves cointegration, 
and has a marginal structure with a (long term) non-stationary factor and a (short term) factor modelling the "stationary" variations.  
In this Subsection we discuss various other particular specifications of these factors beyond classical Ornstein-Uhlenbeck models.

The so-called L\'evy stationary (LS) processes provides us with a flexible class of stationary models
which can be applied as dynamics for $\mathbf X^m=(X_1,\ldots,X_m^{\top})$, $m\in\mathbb N$. To this end, assume 
\begin{equation}
\label{def:ls}
\mathbf X^m(t)=\int_{-\infty}^tG(t-s)\,d\mathbf L(s),
\end{equation}
where $\mathbf L=(L_1,\ldots,L_k)^{\top}$ is a two-sided square integrable $k$-dimensional L\'evy process with zero mean 
and $u\mapsto G(u)$ is a measurable mapping from $\mathbb R_+$ into the space of $m\times k$ matrices with
elements $g_{ij}\in L^2(\R_+), i=1,\ldots,m, j=1,\ldots,k$. 
We remark that the assumption on the $g_{ij}$'s ensures that $\mathbf X^m$ is a well-defined mean zero square integrable stochastic process with values in $\R^m$. LS processes form a subclass of the more general L\'evy semistationary processes considered in e.g. Barndorff-Nielsen, Benth and Veraart~\cite{BNBV}.

As the following Lemma shows, $\mathbf X^m=(X_1,\ldots,X_m)^{\top}$ is strictly stationary:
\begin{lemma}
\label{lemma:strict-stat-ls}
The process $\mathbf X^m=(X_1,\ldots,X_m)^{\top}$ defined in \eqref{def:ls} is a strictly stationary process, that is,
for any $\tau\geq 0$ and $r\in\mathbb N$, the $m\times r$-dimensional random matrices 
$(\mathbf X^m(t_1+\tau),....,\mathbf X^m(t_r+\tau))$ and $(\mathbf X^m(t_1),....,\mathbf X^m(t_r))$ have the same
probability distribution.
\end{lemma}
\begin{proof}
Let $t_{\ell}, \ell=1,\ldots,r$ be an increasing sequence of times on $\R$,
and notice that $(\mathbf X^m(t_1),\mathbf X^m(t_2),\ldots,\mathbf X^m(t_r))^{\top}\in\R^{mr}$. By the independent increment property of
L\'evy processes, it follows that with $\psi$ denoting the cumulant of $\mathbf L(1)$ and $\mathbf z^{\top}=((z^1_1,\ldots,z^m_1),(z_2^1,\ldots,z_2^m),\ldots,(z_r^1,\ldots,z_r^m))\in\R^{mr}$,
\begin{align*}
\E\left[\e^{\mathrm{i}\mathbf z^{\top}(\mathbf X^m(t_1),\ldots,\mathbf X^m(t_r))^{\top}}\right]
&=\E\left[\e^{\mathrm{i}\mathbf z_1^{\top}\int_{-\infty}^{t_1}G(t_1-s)\,d\mathbf L(s)+\cdots\mathrm{i}\mathbf z_r^{\top}\int_{-\infty}^{t_1}G(t_r-s)\,d\mathbf L(s)}\right] \\
&\qquad\times \E\left[\e^{\mathrm{i}\mathbf z_2^{\top}\int_{t_1}^{t_2}G(t_2-s)\,d\mathbf L(s)+\cdots\mathrm{i}\mathbf z_r^{\top}\int_{t_1}^{t_2}G(t_r-s)\,d\mathbf L(s)}\right] \\
&\qquad\times\cdots\times\E\left[\e^{\mathrm{i}\mathbf z_r^{\top}\int_{t_{r-1}}^{t_r}G(t_r-s)\,d\mathbf L(s)}\right] \\
&=\exp\left(\int_{-\infty}^{t_1}\psi\left(G(t_1-s)^{\top}\mathbf z_1+\cdots+G(t_r-s)^{\top}\mathbf z_r\right)\,ds\right) \\
&\qquad\times\exp\left(\int_{t_1}^{t_2}\psi\left(G(t_2-s)^{\top}\mathbf z_2+\cdots+G(t_r-s)^{\top}\mathbf z_r\right)\,ds\right) \\
&\qquad\times\cdots\times\exp\left(\int_{t_{r-1}}^{t_r}\psi(G(t_r-s)^{\top}\mathbf z_r)\,ds\right).
\end{align*}
Here we have used the notation $\mathbf z_i:=(z_i^1,\ldots,z_i^m)^{\top}$.
Thus, 
after a change of variables, we see that the characteristic function of $(\mathbf X^m(t_1),....,\mathbf X^m(t_r))$ depends on 
$(t_2-t_1,t_3-t_2,...,t_r-t_{r-1})$ only, and we can conclude that the probability distribution of 
$(\mathbf X^m(t_1+\tau),....,\mathbf X^m(t_r+\tau))$ equals that of $(\mathbf X^m(t_1),....,\mathbf X^m(t_r))$ for any $\tau>0$. Strict stationarity follows.
\end{proof}
Remark that any linear combination of $X_1,\ldots,X_m$ is strictly stationary whenever 
$(X_1,\ldots,X_m)$ is strictly stationary. If the real-valued process $\{U(t)\}_{t\geq 0}$ is a strictly stationary process, we have that
its probability distribution $P_U(t,\cdot)$ satisfies $P_U(t+\tau,\cdot)=P_U(t,\cdot)$ for all $t\geq 0$, for 
any given $\tau\geq0$. Hence, $P_U(t,\cdot)\equiv P_U(\cdot)$, that is, it is independent of
time $t$. This implies trivially that $P_U(t,\cdot)\rightarrow P_U(\cdot)$ when $t\rightarrow\infty$, and moreover, the characteristic function of $U(t)$ is also independent of $t$. Hence, for any pricing system $(\mathcal P,\mathbf c)$, where $\mathcal P^{\top}\mathbf c\in\mathcal C^m_{\mathbf X}$, we have that $\mathbf c^{\top}\mathcal P\mathbf  X(t)=(\mathbf e_1\mathcal P^{\top}\mathbf c)X_1(t)+\cdots+(\mathbf e_m^{\top}\mathcal P^{\top}\mathbf c)X_m(t)$,
i.e., a linear combination of $X_1, \ldots,X_m$, which is a strictly stationary process. 
We note in passing that if the characteristic
function of a stochastic process $V(t)$ is independent of $t$, it holds that $P_V(t,\cdot)=P_V(\cdot)$. This implies 
stationarity in the sense that the probability distribution of $V(t)$ is invariant of time, however, it does
not necessarily imply {\it strict} stationarity. In Benth~\cite{benth-eberleinfest}, cointegration models based on 
LS processes were proposed and analysed. 

An example of an LS process is given by 
$$
X_i(t)=\eta_i\int_{-\infty}^t\e^{-\alpha_i(t-s)}\,dB_i(s),
$$ 
for $i=1,\ldots,m$ with $\mathbf B=(B_1,\ldots,B_m)^{\top}$ being a two-sided $m$-dimensional Brownian motion (possibly correlated)
and $\alpha_1,\ldots,\alpha_m$, $\eta_1,\ldots,\eta_m$ positive constants. 
Then it holds that the distribution function of 
$(X_1,\ldots,X_m)$ is time invariant and equal to $\mathcal N(0,C)$, with covariance matrix $C\in R^{m\times m}$ having diagonal
elements $\eta_i^2/(2\alpha_i), i=1,\ldots,m$ and off-diagonal elements $\rho_{ij}\eta_i\eta_j/(\alpha_i+\alpha_j)$ for
$\rho_{ij}$ being the correlation coefficient between $B_i$ and $B_j$, $i\neq j$. 
In fact, this example is a particular case of so-called continuous-time autoregressive moving average (CARMA) processes, as we discuss
next.

For $p\in\mathbb N$, define the matrix $A\in\mathbb R^{p\times p}$ as
\begin{equation}
\label{eq:car-matrix}
A=\left[\begin{array}{cccccc} 0 & 1 & 0 & 0 & ... & 0 \\
0 & 0 & 1 & 0 & ... & 0 \\
.. & . & . & . & ... & . \\
.. & . & . & . & ... & . \\
0 & 0 & 0 & 0 & ... & 1 \\
-\alpha_p & -\alpha_{p-1} & -\alpha_{p-2} & -\alpha_{p-3} & ... & -\alpha_1
\end{array}\right]\,,
\end{equation}
for positive constants $\alpha_k$, $k=1,\ldots, p$.
Consider the $p$-dimensional Ornstein-Uhlenbeck process
\begin{equation}
d\mathbf Y(t)=A\mathbf Y(t)\,dt+\mathbf e_p\,dL(t).
\end{equation}
Here we recall that $\{\mathbf e_i\}_{i=1}^p$ are the $p$ canonical basis vectors in $\mathbb R^p$ and $L$ is a (two-sided) real-valued square-integrable 
L\'evy process with zero mean. 
Following Brockwell~\cite{brockwell}, we define a CARMA($p,q$) process $Z$ for $q<p, p,q\in\mathbb N$ by
\begin{equation}
\label{def-carma-1D}
Z(t)=\mathbf b^{\top} \mathbf Y(t),
\end{equation}
for $\mathbf b\in \mathbb R^p$, where $\mathbf b=(b_0,b_1,\ldots, b_q,0,\ldots,0)^{\top}$ and $b_q=1$. We observe that 
for $q=0$, $\mathbf b=\mathbf e_1$ and we say in this case that $Z$ is a continuous-time autoregressive process of order $p$ (a CAR($p$)-process
in short). We suppose that the $p$ eigenvalues of $A$ have negative real part, which yields that $Z$ is strictly stationary with
$$
Z(t)=\int_{-\infty}^t\mathbf b^{\top}\e^{A(t-s)}\mathbf e_p\,dL(s).
$$ 
Thus, with $G(s):=\mathbf b^{\top}\exp(A(s))\mathbf e_p$, a CARMA($p,q$)-process is an example of a real-valued LS-process. 

We want to apply CARMA-processes as factors in a cointegration model (see Comte~\cite{comte} for an extensive analysis of
cointegration based on CARMA processes). To this end, let $\mathbf X$ be an $n$-dimensional process,
and $\mathbf X^m=(X_1,\ldots,X_m)$ for $m<n$ be an $m$-dimensional CARMA-process. A simple way to define such a process is as follows: 
given an $m$-dimensional two-sided square integrable L\'evy process $\mathbf L=(L_1,\ldots,L_m)^{\top}$ with zero mean. For
$i=1,\ldots,m$, let $X_i$ be as in \eqref{def-carma-1D}, that is, a CARMA($p_i,q_i$)-process driven by $L_i$ and with matrix $A_i\in\mathbb R^{p_i\times p_i}$ having eigenvalues with negative real part. In the notation of LS-processes in 
\eqref{def:ls}, this means that the $m\times m$-matrix-valued function $G(u)$ has diagonal elements 
$g_{ii}(u):=\mathbf b_i^{\top}\exp(A_i(s))\mathbf e_{p_i}$ and off-diagonal elements being zero. 
By Lemma~\ref{lemma:strict-stat-ls}, $\mathbf X^m$ is an $m$-dimensional strictly stationary
process.  Benth and Koekebakker~\cite{BK} consider such models in the context of cointegration. We remark in passing that 
CARMA-processes has been applied to model commodities like oil and power (see e.g. Paschke and Prokopczuk \cite{PP} and
Benth, Kl\"uppelberg, M\"uller and Vos~\cite{BKMV}). Multivariate CARMA-processes going beyond the simple specification we 
consider here have been proposed and analysed by Marquardt and Stelzer~\cite{MS}. Their definition will yield an LS-process
\eqref{def:ls} with the matrix-valued function $G$ having non-zero off-diagonal elements. Thus, we do not only have dependency through the L\'evy processes, but also functional dependencies between the coordinates in vector-valued CARMA process.
Such multivariate CARMA processes is further studied by Schlemm and Stelzer~\cite{SchlemmStelz} and Kevei~\cite{Kevei}. Taking these extensions into account, we have a rich class of stationary processes available for cointegration modelling.

It is well-known (see e.g. Benth and \v{S}altyt\.e Benth~\cite{BSB} and Benth, \v{S}altyt\.e Benth and Koekebakker~\cite{BSBK})
that a CARMA($p,q$)-process on a discrete time scale will define an ARMA($p,q$) time series. Furthermore, as is demonstrated
in Aadland, Benth and Koekebakker~\cite{ABK}, the process $X(t)=\int_0^tZ(s)\,ds$, where $Z$ is a CAR($p$)-process, becomes 
a non-stationary process. Hence, it may serve as a non-stationary factor process in modelling the price dynamics $\mathbf S$. Indeed, 
we can use a set of $n-m$ dependent CAR-processes to define non-stationary processes $X_{m+1},\ldots,X_n$ in this way. 
As Aadland, Benth and Koekebakker~\cite{ABK} show, these processes will become integrated autoregressive times series on a discrete
time scale.  Aadland, Benth and Koekebakker~\cite{ABK} model cointegration in a freight rate market using CAR-processes, both
for the stationary and the non-stationary processes.

\section{Forward pricing under cointegration}
\label{sect:forward}


Denote the forward prices at time $t\geq 0$ of
contracts delivering the underlying assets $\mathbf S=(S_1,\ldots, S_d)^{\top}$ at time $T\geq t$ by 
$\mathbf F(t,T):=(F_1(t,T),\ldots,F_d(t,T))^{\top}\in\R^d$. The price vector
of the $d$ assets are defined by $\mathbf S(t)\in\R^d$ in \eqref{eq:basic-def} with $\mathcal P$ being minimal and of full rank. Thus, we suppose an arithmetic model for the spot market. Assume $\mathbb{Q}\sim\mathbb{P}$ is a pricing measure such that $\mathbf X(t)\in\R^n$ is $\mathbb Q$-integrable for all $t>0$. Then, the forward price vector $\mathbf F(t,T)$ is defined as (see Benth, \v{S}altyt\.e Benth
and Koekebakker~\cite{BSBK}),
\begin{equation}
\mathbf F(t,T)=\mathbb{E}_{\mathbb Q}[\mathbf S(T)\,|\,\mathcal F_t].
\end{equation} 
Hence, by the definition of $\mathbf S$ we find that
\begin{equation}
\mathbf F(t,T)=\mathcal P\mathbb{E}_{\mathbb{Q}}[\mathbf X(T)\,|\,\mathcal F_t]\,.
\end{equation}
To proceed our analysis, the following definition of affinity is convenient:
\begin{definition}
The stochastic process $\{\mathbf X(t)\}_{t\geq 0}$ is said to be {\it affine with respect to $\mathbb Q$},
or $\mathbb Q$-affine for short, if there exist measurable deterministic functions $(t,T)\mapsto\mathcal A(t,T)\in\R^{n\times n}$ and
$(t,T)\mapsto \mathbf a(t,T)\in\R^d$ such that
$$
\mathbb{E}_{\mathbb{Q}}[\mathbf X(T)\,|\,\mathcal F_t]=\mathcal A(t,T)\mathbf X(t)+\mathbf a(t,T)
$$
for $0\leq t\leq T<\infty$.   
\end{definition}
A trivial example of a $\mathbb Q$-affine process $\mathbf X$ is the case when $\mathbf X$ is an $n$-dimensional
$\mathbb Q$-Brownian motion $\mathbf B=(B_1,\ldots,B_n)^{\top}$. Then $\mathbf a=0$, and $\mathcal A$ is the covariance matrix with elements $\rho_{ij}t$ for $\rho_{ii}=1$ and $\rho_{ij}$ being the correlation between $B_i$ and $B_j$, $i\neq j$. 
A less trivial example is provided by Ornstein-Uhlenbeck processes. We show next the affinity property for $\mathbb Q$-semimartingales which are polynomial processes
(see e.g. Cuchiero, Keller-Ressel and Teichmann~\cite{CKRT} and Filipovic and Larsson~\cite{FL} for a definition and analysis of polynomial processes).
\begin{proposition}
\label{prop:polynomial}
Assume the $\mathbb Q$-dynamics of $\mathbf X$ is a polynomial process in  $\R^n$. Then $\mathbf X$ is $\mathbb Q$-integrable and 
$\mathbb Q$-affine, with the functions $(t,T)\mapsto \mathbf a(t,T)$ and $(t,T)\mapsto\mathcal A(t,T)$ being homogeneous, i.e., 
$\mathcal A(t,T)=\mathcal A(T-t)$ and $\mathbf a(t,T)=\mathbf a(T-t)$ (with a slight abuse of notation).
\end{proposition}
\begin{proof}
A polynomial process has finite moments (Lemma~2.17 in Cuchiero, Keller-Ressel and Teichmann~\cite{CKRT}), and thus 
$\mathbf X$ is $\mathbb Q$-integrable. Following the definition of a polynomial process (see e.g. Cuchiero, Keller-Ressel and Teichmann~\cite{CKRT}
and Filipovic and Larsson~\cite{FL}), 
we know that for the generator $\mathcal G$ of $\mathbf X$, there exists a matrix $G\in\R^{n\times n}$ and a vector
$\mathbf b\in\R^n$ such that
$\mathcal G\mathbf x=G\mathbf x+\mathbf b$. This holds true since $\mathbf x$ is a first order polynomial and the generator is
preserving the order when applied to polynomials. Therefore, from the martingale problem of polynomial processes,
$$
\E_{\mathbb Q}[\mathbf X(T)\,|\,\mathcal F_t]
=\mathbf X(t)+\int_t^T \big( G\E_{\mathbb Q}[\mathbf X(s)\,|\,\mathcal F_t]+\mathbf b\big)\,ds.
$$ 
and thus,
$$
\E_{\mathbb Q}[\mathbf X(T)\,|\,\mathcal F_t]=\e^{G(T-t)}\mathbf X(t)+\int_t^T\e^{G(T-s)}\mathbf b\,ds.
$$ 
The result follows. 
\end{proof}
We remark in passing that the class of polynomial processes has a much richer structure than really needed for the $\mathbb Q$-affinity. 
The generator of a polynomial process preserves the order of any polynomial, while $\mathbb Q$-affinity only requires that the generator 
preserves the first order polynomials. 

As an example, consider an Ornstein-Uhlenbeck process in $\mathbb R^n$ with $\mathbb Q$-dynamics 
$$
d\mathbf X(t)=(\boldsymbol{\mu}+C\mathbf X(t))\,dt+\Sigma\,d\mathbf W(t).
$$
Here, $\boldsymbol{\mu}\in\R^n$, $C\in\R^{n\times n}$, $\Sigma\in\R^{n\times m}$ and $\mathbf W$ is an $m$-dimensional Brownian motion. A direct calculation reveals that for $t\leq T$,
$$
\mathbf X(T)=\e^{C(T-t)}\mathbf X(t)+\int_t^T\e^{C(T-s)}\boldsymbol{\mu}\,ds+\int_t^T\e^{C(T-s)}\Sigma\,d\mathbf W(s),
$$
and thus
$$
\E_{\mathbb{Q}}[\mathbf X(T)\,|\,\mathcal F_t]=\e^{C(T-t)}\mathbf X(t)+\int_0^{T-t}\e^{Cs}\boldsymbol{\mu}\,ds.
$$
In conclusion, $\mathbb{Q}$-affinity holds with $\mathcal A(T-t)=\exp(C(T-t))$ and $\mathbf a(T-t)=\int_0^{T-t}\exp(Cs)\boldsymbol{\mu}\,ds$. Note that both $\mathcal A$ 
and $\mathbf a$ are homogeneous in time. Whenever $C$ is an invertible matrix, we find
$$
\mathbf a(T-t)=C^{-1}(\e^{C(T-t)}-I)\boldsymbol{\mu}
$$ 
where $I\in\R^{n\times n}$ is the identity matrix. 

We have the following simple result:
\begin{corollary}
If $\{\mathbf X(t)\}_{t\geq 0}$ is $\mathbb Q$-integrable and $\mathbb Q$-affine process in $\R^n$, then 
$$
\mathbf F(t,T)=\mathcal P\mathcal A(t,T)\mathbf X(t)+\mathcal P \mathbf a(t,T)\,.
$$
Moreover, if $\mathcal P\mathcal A(t,T)=\widetilde{\mathcal A}(t,T)\mathcal P$ for some $\widetilde{\mathcal A}(t,T)\in\R^{d\times d}$, $0\leq t\leq T<\infty$, then $\mathbf F(t,T)=\widetilde{\mathcal A}(t,T)\mathbf S(t)+\mathcal P\mathbf a(t,T)$ (i.e., the forward price vector is affine in the asset price $\mathbf S$.) 
\end{corollary}
\begin{proof}
This is trivial from the definition of affinity. 
\end{proof}
We note that in the case $d=n$, we have forward prices which are affine in the underlying spot when
$\mathcal P$ and $\mathcal A(t,T)$ commutes for all $t\leq T$.

Let us next turn to the question of cointegration in the forward market. As $t\leq T<\infty$, it is natural to switch to the Musiela parametrization, and express forward prices in terms of {\it time to maturity} $x:=T-t$ rather than {\it time of maturity} $T$. I.e., introduce the random fields $\mathbf f(t,x)$ 
for $x\geq 0$ by
\begin{equation}
\mathbf f(t,x):=\mathbf F(t,t+x)\,.
\end{equation}
Hence, we find in the case of $\{\mathbf X(t)\}_{t\geq 0}$ being $\mathbb Q$-affine and $\mathbb Q$-integrable that
$$
\mathbf f(t,x)=\mathcal P\mathcal A(t,t+x)\mathbf X(t)+\mathbf a(t,t+x)\,.
$$
The following Proposition holds:
\begin{proposition}
\label{prop:coint-futures}
Fix $x\geq 0$, and suppose that $\mathbf X$ is $\mathbb Q$-integrable and $\mathbb Q$-affine, with $\mathcal A$ and 
$\mathbf a$ homogeneous (e.g., $\mathcal A(t,T)=\mathcal A(T-t)$ and $\mathbf a(t,T)=\mathbf a(T-t)$). Then 
$t\mapsto \mathbf f(t,x)$ is cointegrated if
there exists a vector $\mathbf c\in\R^d$ such that $\mathbf c^{\top}\mathcal P\mathcal A(x)\in\mathcal C_{\mathbf X}$, or, equivalently, 
$(\mathcal P\mathcal A(x),\mathbf c)$ is
a cointegrated pricing system.
\end{proposition}
\begin{proof}
This follows readily from the definitions and the fact that for homogeneous $\mathcal A$ and $\mathbf a$, 
$\mathcal A(t,t+x)=\mathcal A(x)$ and $\mathbf a(t,t+x)=\mathbf a(x)$. 
\end{proof}
\begin{remark}
We emphasise that $x\geq 0$ is fixed in Proposition~\ref{prop:coint-futures}. This means that it is the dynamics of
the forward contracts with fixed time to maturity that is cointegrated. This can be viewed as a roll-over contract, where
one fixes the time to maturity and "rolls over" the position when time progresses. The actual forward price dynamics
will in general {\it not} be cointegrated as it will depend on $\mathcal A(t,T)$ and $\mathbf a(t,T)$, which varies with time $t$. Benth and Koekebakker~\cite{BK} make a similar observation for a more
particular HJM-type cointegrated forward price model.   
If $x=0$, or equivalently $t=T$, we are back to the spot price case. Propositions~\ref{prop:coint-futures} and 
\ref{prop:polynomial} show that polynomial processes can be used to build cointegrated forward price models. 
\end{remark}

Consider the case when the $\mathbb{Q}$-dynamics of $\mathbf X(t)\in\R^3$ is such that $X_3(t)$ is a non-stationary process 
and $(X_1,X_2)$ admits a limiting distribution. From previous considerations we then have that
$\mathcal C_{\mathbf X}=\{\mathbf a\in\R^3\,|\,a_3=0\}$. In the context of Example~\ref{motivating-ex}, for any 
pricing matrix $\mathcal P\in\R^{2\times 3}$ and $\mathbf c\in\R^2$, we find that
$\mathbf c^{\top}\mathcal P\in\mathcal C_{\mathbf X}$ if and only if $\mathbf c^{\top}\mathcal P\mathbf e_3=0$ (e.g, the third coordinate of
$\mathbf c^{\top}\mathcal P$ is equal to zero). With $p_{ij}$ denoting the $ij$th element of $\mathcal P$, we find that 
$\mathbf c^{\top}\mathcal P\in\mathcal C_{\mathbf X}$ if and only if $c_1p_{13}+c_2p_{23}=0$. Let us analyse this for non-trivial
$\mathbf c$ (e.g., $\mathbf c\neq\mathbf 0$) and $\mathcal P$ being minimal. Minimality of $\mathcal P$ means that 
$\mathcal P\mathbf e_i\neq(0,0)^{\top}$ for $i=1,2,3$, and in particular for $i=3$ we find $(p_{13},p_{23})\neq (0,0)$. 
Thus, we find that $(\mathcal P,\mathbf c)$ is a cointegrated pricing system if and only if either $c_2,p_{13}\neq 0$ and
$c_1/c_2=-p_{23}/p_{13}$ or $c_1,p_{23}\neq 0$ and $c_2/c_1=-p_{13}/p_{23}$. 
If $\mathbf X$ is $\mathbb Q$-affine with a matrix $\mathcal A(t,T)=\mathcal A(T-t)\in\R^{3\times 3}$ satisfying 
$\mathbf e_1^ {\top}\mathcal A(x)\mathbf e_3=\mathbf e_2^{\top}\mathcal A(x)\mathbf e_3=0$ yields
that $\mathbf c^{\top}\mathcal P\mathcal A(x)\in\mathcal C_{\mathbf X}$ for any cointegrated pricing system 
$(\mathcal P,\mathbf c)$. 

As a particular case of the above, consider the factor process
\begin{equation}
\label{eq:OU-example}
d\mathbf X(t)=\left(\boldsymbol \mu+\left[\begin{array}{cc} C & \mathbf{0} \\ \mathbf{0}^{\top} & 0\end{array}\right]\mathbf X(t)\right)\,dt+\Sigma\,d\mathbf W(t)
\end{equation}
with $\Sigma,C\in\R^{2\times2}$, $\boldsymbol\mu\in\R^3$ and $\mathbf{0}=(0,0)^{\top}$. Further, 
$\mathbf W$ is assumed to 
be a trivariate $\mathbb Q$-Brownian motion.
Here, $(X_1,X_2)$ will be a bivariate OU process
with mean-reversion matrix $C$ and noise vector $(\mathbf e_1^{\top}\Sigma\,d\mathbf W(t),\mathbf e_2^{\top}\Sigma\,d\mathbf W(t))^{\top}$,
which admits a limiting distribution whenever $C$ has eigenvalues with negative real part.  
The process $X_3$ is a drifted Brownian motion. Then,
$$
\E_{\mathbb Q}[\mathbf X(T)\,|\,\mathcal F_t]=\mathcal A(T-t)\mathbf X(t)+\mathbf a(T-t),
$$ 
where
\begin{equation}
\label{semigroup-ou-example}
\mathcal A(T-t)=\left[\begin{array}{cc} e^{C(T-t)} & \mathbf{0} \\ \mathbf{0}' & 1 \end{array}\right],
\end{equation}
and $\mathbf a(T-t)=\int_0^{T-t}\mathcal A(y)\boldsymbol \mu\,dy$.
Thus, for $(\mathcal P,\mathbf c)$ being a cointegrated pricing system, the forward prices will also be cointegrated.
We see that we obtain cointegration both for the spot (under $\mathbb Q$) and the forward prices
for rather general models of $\mathbf X$, including a full corrrelation structure between the three noises
$\mathbf W$ and flexible mean reversion matrix $C$.

Note that if $\mathcal A$ is not homogeneous, that is, $\mathbf X$ is $\mathbb Q$-affine for a non-homogeneous
$\mathcal A$, then we may lose cointegration in the forward process. In the case $\mathbf a$ is non-homogeneous, 
we may recover cointegration as long as $\mathcal A$ is homogeneous by considering the "de-trended" forward price vector
$\bar{\mathbf f}(t,x):=\mathbf f(t,x)-\mathcal P\mathbf a(t,x)$. In 
that case, $\bar{\mathbf f}(t,x)$ is cointegrated whenever $\mathbf c^{\top}\mathcal P\mathcal A(x)\in\mathcal C_{\mathbf X}$. Indeed, this is a relevant case for commodity markets with seasonally varying prices. For example, in power markets, where
prices are highly influenced by weather conditions, it may appear that $\mathbf a$ is not homogeneous. Indeed, the factor model
\eqref{eq:OU-example} can be used as a model for spot prices with $\boldsymbol{\mu}$ being time dependent, i.e.
$t\mapsto\boldsymbol{\mu}(t)$ for some measurable real-valued function being bounded on compacts. Then 
$\mathbf a(t,T)=\int_t^{T}\mathcal A(T-s)\boldsymbol{\mu}(s)\,ds$ is not in general homogeneous. Typically, $\boldsymbol{\mu}(t)$ models a seasonal mean price, towards which the stationary part of $\mathbf X$ mean reverts (see e.g. Benth, \v{S}altyte Benth and Koekebakker~\cite{BSBK} for models of this type with seasonality).  

\subsection{General LS-processes}
In general, LS-process will not be $\mathbb Q$-affine. In this subsection we analyse forward pricing involving
LS-processes. 

For $x\geq 0$, we define the $\R^m$-valued random field $\widetilde{\mathbf X}^m(t,x)$ by
\begin{equation}
\widetilde{\mathbf X}^m(t,x):=\int_{-\infty}^tG(t-s+x)\,d\mathbf L(s),
\end{equation}
with $G$ and $\mathbf L$ being as in the definition of the LS-process in \eqref{def:ls}. We assume that this is
the $\mathbb Q$-dynamics of $\widetilde{\mathbf X}^m(t,x)$. In particular, for
$x=0$, we are back to $\mathbf X^m(t)$ as in \eqref{def:ls} (but now considered as a dynamics
with respect to $\mathbb Q$). Moreover, following the proof of 
Lemma~\ref{lemma:strict-stat-ls}, the stochastic process $t\mapsto \widetilde{\mathbf X}^m(t,x)$ is strictly stationary for every
$x\geq 0$. It is simple to see that
\begin{equation}
\E_{\mathbb Q}[\mathbf X^m(T)\,|\,\mathcal F_t]=\E_{\mathbb Q}[\widetilde{\mathbf X}^m(T,0)\,|\,\mathcal F_t]=\widetilde{\mathbf X}^m(t,T-t),
\end{equation}
by appealing to the independent increment property of L\'evy processes. Hence, assuming a factor process
$\mathbf X$ which is $\mathbb Q$-integrable, where $\mathbf X^m$ is given by an LS-process as in \eqref{def:ls} with respect 
to the probability $\mathbb Q$, we find that
\begin{equation}
\mathbf f(t,x)=\mathcal P\left(\widetilde{X}_1(t,x),\ldots,\widetilde{X}_m(t,x),\E_{\mathbb Q}[X_{m+1}(T)\,|\,\mathcal F_t],\ldots,\E_{\mathbb Q}[X_{n}(T)\,|\,\mathcal F_t]\right)^{\top}.
\end{equation}
We see that any $\mathbf c\in\R^d$ such that $\mathbf c^{\top}\mathcal P\in\mathcal C^m_{\mathbf X}$ implies that 
$\mathbf c^{\top}\mathbf f(t,x)$ becomes a linear combination of $\widetilde{X}_1(t,x),\ldots,\widetilde{X}_m(t,x)$, and therefore strictly stationary.
Hence, $\mathbf c$ will be a cointegration vector for $\mathbf f(t,x)$.

The classes of CARMA-processes and their multivariate extensions discussed in the previous section provide a rich class of LS-processes that can be used for modelling cointegrated forward prices under the Musiela parametrization.

\subsection{Factor models of geometric type}

Classically, pricing models in finance have been geometric. In our context, we recall from \eqref{def:geom-factor-model} that this means a spot price dynamics $\mathbf S$ of the form $\ln\mathbf S(t)=\mathcal P\mathbf X(t)$. The forward price vector
$\mathbf F(t,T)=(F_1(t,T),\ldots,F_d(t,T))^{\top}$ will be given by
\begin{equation}
\label{def:exp-forward}
F_i(t,T)=\E_{\mathbb Q}\left[\exp(\mathbf e_i^{\top}\mathcal P\mathbf X(T))\,|\,\mathcal F_t\right]
\end{equation}
for $t\leq T$ and $i=1,\ldots,d$. We recall that $\{\mathbf e_i\}_{i=1}^d$ are the canonical basis vectors in $\R^d$, thus $\mathbf e_i^{\top}\mathcal P\mathbf X(T)$ is the $i$th coordinate of the vector $\mathcal P\mathbf X(t)$, i.e., $\ln S_i(t)$. 
We are naturally led to define the following class of factor processes:
\begin{definition}
A process $\mathbf X$ is called exponentially $\mathbb Q$-affine if for every $\mathbf z\in\R^n$, $\mathbf z^{\top}\mathbf X(T)$
has finite exponential moment under $\mathbb Q$ and
there exist measurable mappings $(t,T)\mapsto\boldsymbol\alpha(t,T;\mathbf z)\in\R^{n}$ and 
$(t,T)\mapsto a(t,T;\mathbf z)\in\R$ such that 
$$
\E_{\mathbb Q}[\exp(\mathbf z^{\top}\mathbf X(T))\,|\,\mathcal F_t]=\exp\left(\boldsymbol\alpha(t,T;\mathbf z)^{\top}\mathbf X(t)+a(t,T;\mathbf z)\right)
$$
for all $t\leq T$.
\end{definition} 

For exponential $\mathbb Q$-affine factor processes, we have:
\begin{proposition}
If $\mathbf X$ is an $n$-dimensional exponential $\mathbb Q$-affine factor process, then $\mathbf F(t,T), t\leq T$ has coordinates
$$
F_i(t,T)=\exp\left(\boldsymbol\alpha(t,T;\mathcal P^{\top}\mathbf e_i)^{\top}\mathbf X(t)+a(t,T;\mathcal P^{\top}\mathbf e_i)\right)
$$
for $i=1,\ldots,d$.
\end{proposition}
\begin{proof}
This follows immediately from the definition of exponential affinity and \eqref{def:exp-forward}. 
\end{proof}
For example, if $\mathbf X$ is given by \eqref{eq:OU-example} (under $\mathbb Q$), it will be exponentially $\mathbb Q$-affine, 
as the following Lemma shows:
\begin{lemma}
Suppose that $\mathbf X$ is the factor process in $\R^3$ defined in \eqref{eq:OU-example}. Then $\mathbf X$ is
exponential $\mathbb Q$-affine, with $\boldsymbol\alpha(t,T;\mathbf z)=\mathcal A(T-t)^{\top}\mathbf z$ and
$$
a(T-t;\mathbf z)=\int_0^{T-t}\mathbf z^{\top}\mathcal A(s)\boldsymbol\mu+\mathbf z^{\top}\mathcal A(s)\Sigma C\Sigma^{\top}\mathcal A(s)^{\top}\mathbf z\,ds.
$$
Here, $\mathcal A$ is defined in \eqref{semigroup-ou-example} and $C$ is the $3\times 3$ covariance matrix
of $\mathbf W$. 
\end{lemma}
\begin{proof}
It holds that
$$
\mathbf X(T)=\mathcal A(T-t)\mathbf X(t)+\int_0^{T-t}\mathcal A(s)\boldsymbol\mu\,ds+\int_t^T\mathcal A(T-s)\Sigma\,d\mathbf W(s),
$$
where $\mathcal A(s)$ is defined in \eqref{semigroup-ou-example}. As the stochastic integral on the right hand side is a Wiener
integral, it is a Gaussian random variable and hence $\mathbf z^{\top}\mathbf X(t)$ has finite exponential moment for every
$\mathbf z\in\R^n$. By the independent increment property of Brownian motion and the $\mathbf X(t)$ being
$\mathcal F_t$-measurable, we find
\begin{align*}
\E_{\mathbb Q}\left[\exp(\mathbf z^{\top}\mathbf X(T))\,|\,\mathcal F_t\right]&=\exp\left(\mathbf z^{\top}\mathcal A(T-t)\mathbf X(t)+\int_0^{T-t}\mathbf z^{\top}\mathcal A(s)\boldsymbol\mu\,ds\right) \\
&\qquad\qquad\times\E_{\mathbb Q}\left[\exp\left(\int_t^T\mathbf z^{\top}\mathcal A(T-s)\Sigma\,d\mathbf W(s)\right)\right] \\
&=\exp\left(\mathbf z^{\top}\mathcal A(T-t)\mathbf X(t)+\int_0^{T-t}\mathbf z^{\top}\mathcal A(s)\boldsymbol\mu\,ds\right) \\
&\qquad\qquad\times\exp\left(\int_t^T\mathbf z^{\top}\mathcal A(T-s)\Sigma C\Sigma^{\top}\mathcal A(T-s)^{\top}\mathbf z\,ds\right).
\end{align*}
The result follows.
\end{proof}
We remark in passing that one can easily extend the above Lemma to higher dimensions than 3. Observe that both $\boldsymbol\alpha$ and $a$ are homogeneous, i.e., depending only on the time to maturity $T-t$. Let $(\mathcal P,\mathbf c)$ for
$\mathbf c\in\R^2$ and $\mathcal P\in R^{2\times3}$ be a cointegrated pricing system (under $\mathbb Q$), i.e., $\mathbf c^{\top}\mathcal P\in\mathcal C_{\mathbf X}$. We have then
$$
\ln f_i(t,x)=\mathbf e_i^{\top}\mathcal P\mathcal A(x)\mathbf X(t)+a(x;\mathcal P^{\top}\mathbf e_i)
$$
for $i=1,2$. Moreover, as we have seen earlier, $\mathbf c^{\top}\mathcal P\mathcal A(x)\in\mathcal C_{\mathbf X}$, and therefore
the logarithmic forward prices $f_1(t,x)$ and $f_2(t,x)$ are cointegrated for the cointegration vector $\mathbf c$. 
 
Next, let us focus on general LS-processes as factors in a geometric model. Suppose that $\mathbf X^m$ has $\mathbb Q$-dynamics
defined as in \eqref{def:ls}. We find the following:
\begin{proposition}
\label{prop:exp-ls-forward}
Assume $\mathbf X^m$ is an $m$-dimensional process with $\mathbb Q$-dynamics as in \eqref{def:ls}. Let $t\leq T$. If $\mathbf z^{\top}\mathbf X^m(t)$ has finite exponential moment under $\mathbb Q$ for $\mathbf z\in\R^m$, then 
\begin{align*}
\E_{\mathbb Q}&\left[\exp\left(\mathbf z^{\top}\int_{-\infty}^TG(T-s)\,d\mathbf L(s)\right)\,|\,\mathcal F_t\right] \\
&=\exp\left(\mathbf z^{\top}\int_{-\infty}^tG(T-s)\,d\mathbf L(s)+\int_0^{T-t}\psi_{\mathbb Q}(G(s)^{\top}\mathbf z)\,ds\right)
\end{align*}
where $\psi_{\mathbb Q}$ is the cumulant of $\mathbf L(1)$ under $\mathbb Q$. 
\end{proposition}
\begin{proof}
Since $\int_{-\infty}^tG(T-s)\,d\mathbf L(s)$ is $\mathcal F_t$-measurable, and the L\'evy process $\mathbf L$ has independent increments, it follows that 
\begin{align*}
\E_{\mathbb Q}&\left[\exp\left(\mathbf z^{\top}\int_{-\infty}^TG(T-s)\,d\mathbf L(s)\right)\,|\,\mathcal F_t\right] \\
&=\exp\left(\mathbf z^{\top}\int_{-\infty}^tG(T-s)\,d\mathbf L(s)\right)\E_{\mathbb Q}\left[\exp\left(\int_t^T\mathbf z^{\top}
G(T-s)\,d\mathbf L(s)\right)\right] \\
&=\exp\left(\mathbf z^{\top}\int_{-\infty}^tG(T-s)\,d\mathbf L(s)+\int_t^T\psi_{\mathbb Q}\left(G(T-s)^{\top}
\mathbf z\right)\,ds\right).
\end{align*}
and the result follows. 
\end{proof}
Express the factor process as $\mathbf X=(\mathbf X^m,\widehat{\mathbf X})\in\R^n$, for $\widehat{\mathbf X}$ being
a process in $\R^{n-m}$, $n>m\in\mathbb N$. We further suppose that $\mathbf X^m$ is an LS-process under 
$\mathbb Q$, as in \eqref{def:ls}. Consider a cointegrated pricing system $(\mathcal P,\mathbf c)$, that is,
$\mathcal P^{\top}\mathbf c\in\mathcal C_{\mathbf X}$, and introduce the following representation of the $d\times n$-matrix
$\mathcal P$: let $\mathcal P^m\in \R^{d\times m}$ and $\widehat{\mathcal P}\in\R^{d\times(n-m)}$ be such that 
$\mathcal P=[\mathcal P^m\quad\widehat{\mathcal P}]$. Then for $i=1,\ldots,d$, 
$$
\mathbf e_i^{\top}\mathcal P\mathbf X(T)=\mathbf e_i^{\top}\mathcal P^m\mathbf X^m(T)+\mathbf e_i^{\top}
\widehat{\mathcal P}\widehat{\mathbf X}(T).
$$ 
Assume that $\mathbf e_i^{\top}\mathcal P^m\mathbf X^m(T)$ and $\mathbf e_i^{\top}
\widehat{\mathcal P}\widehat{\mathbf X}(T)$ have finite exponential moment under $\mathbb Q$, and that they are
conditionally independent with respect to $\mathcal F_t$ for all $t\leq T$. Then it holds for $i=1,\ldots,d$ and $t\leq T$ that 
\begin{align*}
F_i(t,T)&=\E_{\mathbb Q}\left[\exp\left(\mathbf e_i^{\top}\mathcal P\mathbf X(T)\right)\,|\,\mathcal F_t\right] \\
&=\E_{\mathbb Q}\left[\exp\left(\mathbf e_i^{\top}\mathcal P^m\mathbf X^m(T)\right)\,|\,\mathcal F_t\right]
\E_{\mathbb Q}\left[\exp\left(\mathbf e_i^{\top}
\widehat{\mathcal P}\widehat{\mathbf X}(T)\right)\,|\,\mathcal F_t\right] \\
&=\exp\left(\mathbf e_i^{\top}\mathcal P^m\int_{-\infty}^tG(T-s)\,d\mathbf L(s)+\int_0^{T-t}\psi_{\mathbb Q}\left(G(s)^{\top}\mathcal P^{m,\top}\mathbf e_i\right)\,ds\right) \\
&\qquad\qquad\times\E_{\mathbb Q}\left[\exp\left(\mathbf e_i^{\top}
\widehat{\mathcal P}\widehat{\mathbf X}(T)\right)\,|\,\mathcal F_t\right] 
\end{align*}
where we used Prop.~\ref{prop:exp-ls-forward} with $\mathbf z=\mathcal P^{m,\top}\mathbf e_i$in the last equality. We find that
\begin{align*}
\ln f_i(t,x)&=\mathbf e_i^{\top}\mathcal P^m\int_{-\infty}^tG(t-s+x)\,d\mathbf L(s)+\int_0^{x}\psi_{\mathbb Q}\left(G(s)^{\top}\mathcal P^{m,\top}\mathbf e_i\right)\,ds \\
&\qquad\qquad+\ln\E_{\mathbb Q}\left[\exp\left(\mathbf e_i^{\top}
\widehat{\mathcal P}\widehat{\mathbf X}(t+x)\right)\,|\,\mathcal F_t\right],
\end{align*}
for $i=1,\ldots,d$ and $x\geq 0$. The last term is nonlinear in the vector $\widehat{\mathbf X}$, and $\mathbf c$ may fail to
be a cointegration vector for $\ln\mathbf f(t,x)$, even in the case when $\mathcal P^{\top}\mathbf c\in\mathcal C^m_{\mathbf X}$.  

However, typically in applications, $\widehat{\mathbf X}=\mathbf U$ for a L\'evy process $\mathbf U$ in $\R^{n-m}$. In the simplest case,
$\mathbf U(t)=\boldsymbol{\mu}t+\Sigma\,d\mathbf W(t)$ for $\boldsymbol\mu\in\R^{n-m}$, $\Sigma$ an $(n-m)\times(n-m)$
volatility matrix and $\mathbf W$ a $\mathbb Q$-Brownian motion in $\R^{n-m}$. Suppose that $\mathbf U$ is independent of
$\mathbf L$. Then it follows that $\mathbf e_i^{\top}\mathcal P^m\mathbf X^m(T)$ and $\mathbf e_i^{\top}
\widehat{\mathcal P}\widehat{\mathbf X}(T)$ are
conditionally independent with respect to $\mathcal F_t$ for all $t\leq T$. Moreover, $\mathbf e_i^{\top}
\widehat{\mathcal P}\widehat{\mathbf X}(T)$ have finite exponential moment under $\mathbb Q$ when $\mathbf U$ is a
drifted Brownian motion as exemplified above. Without any loss of generality, 
we assume that the coordinates $W_i$, $i=1,\ldots n-m$, of $\mathbf W$ are independent. Denoting $\kappa_{\mathbb Q}$
the cumulant function of $\mathbf U$, we find by resorting to the independent increment property of L\'evy processes
that
$$
\E_{\mathbb Q}\left[\exp\left(\mathbf e_i^{\top}
\widehat{\mathcal P}\widehat{\mathbf X}(t+x)\right)\,|\,\mathcal F_t\right]=
\exp\left(\mathbf e_i^{\top}\widehat{\mathcal P}\widehat{\mathbf X}(t)+x\kappa_{\mathbb Q}(\widehat{\mathcal P}^{\top}\mathbf e_i))\right).
$$  
In this case we have that 
\begin{equation}
\ln\mathbf f(t,x)=\mathbf c^{\top}\mathcal P\left(\begin{array}{c}\widetilde{\mathbf X}^m(t,x) \\ \widehat{\mathbf X}(t)\end{array}\right)+
\mathbf h(x)
\end{equation}
with $\mathbf h(x)\in\R^d$ with coordinates
\begin{equation}
h_i(x)=\int_0^{x}\psi_{\mathbb Q}\left(G(s)^{\top}\mathcal P^{m,\top}\mathbf e_i\right)\,ds+x\kappa_{\mathbb Q}(\widehat{\mathcal P}^{\top}\mathbf e_i)).
\end{equation}
After a simple modification of Lemma~\ref{lemma:strict-stat-ls}, we know that $\widetilde{\mathbf X}^m(t,x)$ is a strictly stationary process in $\R^m$. In this case, any $\mathbf c\in\R^d$ such that $\mathcal P^{\top}\mathbf c\in\mathcal C_{\mathbf X}^m$ is a cointegration vector
for $\ln\mathbf f(t,x)$. Thus, the cointegration vector for the spot yields cointegration of the forwards as well.

\subsection{Market probability $\mathbb P$ vs. pricing measure $\mathbb Q$}

Throughout this Section we have assumed a factor process specified directly under the pricing measure $\mathbb Q$ in our analysis of cointegration for forward markets. Indeed, we have supposed a cointegrated spot model {\it under the pricing measure} $\mathbb Q$
rather than under the market probability $\mathbb P$. In practice, the situation is more likely that one has a cointegrated spot
model {\it under the market probability} $\mathbb P$, and introduces a pricing measure $\mathbb Q$ to price forwards on the spot
prices. The next step is to analyse possible cointegration of the forward prices.

A common approach in commodity and energy markets for introducing a pricing measure $\mathbb Q$ is to consider structure preserving
equivalent probabilites (see Benth, \v{S}altyt\.e Benth and Koekebakker~\cite{BSBK}, Benth et al.~\cite{BKMV}, Benth and Koekebakker~\cite{BK}, Eydeland and Wolyniec~\cite{EW}, Geman~\cite{Geman}, Lucia and Schwartz~\cite{LS}, to
mention just a few). By this we mean a probability $\mathbb Q\sim\mathbb P$ that preserves the probabilistic structure of the factor process $\mathbf X$. In commodity markets, one typically chooses the Esscher
and Girsanov transforms as the approach to construct pricing measures, with constant market price of risk (see Benth, \v{S}altyt\.e Benth and Koekebakker \cite{BSBK} for an introduction of the Esscher transform in commodity markets and, e.g., Karatzas and Shreve~\cite{KS} for a general analysis of the Girsanov transform). Roughly speaking, any cointegrated pricing system $(\mathcal P,\mathbf c)$ for $\mathbb P$ will also
become a cointegrated pricing system for $\mathbb Q$ when we use the Esscher and Girsanov transform with constant market price
of risk to the factor process. We emphasise that to apply these transforms, we need to have a factor process driven by L\'evy processes
with finite exponential moments of some order.   

For CARMA processes driven by Brownian motion one can introduce pricing measures that are structure
preserving, where the coefficients $\alpha_i, i=1,\ldots,p$ in the CARMA matrix $A$ in \eqref{eq:car-matrix} is changed 
(see Benth and \v{S}altyt\.e Benth~\cite{BSB}). Restricting to Ornstein-Uhlenbeck processes, one can find a similar
structure preserving pricing measure which slows down the speed of mean reversion, even for processes driven by positive L\'evy processes
(see Benth and Ortiz-Latorre~\cite{BOL}). Thus, we see that for a rich class of CARMA and Ornstein-Uhlenbeck processes, we can 
introduce pricing measures $\mathbb Q$ that preserves stationarity of the factor process $\mathbf X^m$. Combined with an Esscher transform for the remaining $\widetilde{\mathbf X}$, where $\mathbf X=(\mathbf X^m,\widetilde{\mathbf X})$, we see that the essential
probabilistic characteristics for cointegration under $\mathbb P$ is transferred to $\mathbb Q$. In conclusion, we can obtain 
cointegration in the spot market under $\mathbb P$ which is transferred to $\mathbb Q$. In this way, our analysis 
of cointegration in the forward market in this Section can be linked to practice.   

In general, there is no equivalence of cointegrated pricing systems under the market probability $\mathbb P$ and the chosen
pricing measure $\mathbb Q$. For example, considering the measure change in Benth and Ortiz-Latorre~\cite{BOL} which is reducing
the speed of mean reversion of an Ornstein-Uhlenbeck process, we can in fact "kill" the mean reversion, and turn the 
stationary Ornstein-Uhlenbeck dynamics under $\mathbb P$ into a non-stationary dynamics under $\mathbb Q$. Such a situation may also occur in the case of CARMA processes using the measure change suggested in Benth and \v{S}altyt\.e Benth~\cite{BSB}. We see
that we may alter the space of possible cointegration pricing systems when going from $\mathbb P$ to $\mathbb Q$. For example, considering the simple three-factor model by Lucia and Schwartz in Example~\ref{motivating-ex}, by introducing a measure change
as in Benth and Ortiz-Latorre~\cite{BOL} which kills the mean reversion of $X_1$ and $X_2$, we end up with three non-stationary
(indeed, drifted Brownian motions) processes under $\mathbb Q$. In this case, the only cointegrated pricing systems under 
$\mathbb Q$ are those $(\mathcal P,\mathbf c)$ for which $\mathcal P^{\top}\mathbf c=\mathbf 0\in\R^3$. Interestingly, for this example, 
we can price forwards and recover cointegration for the forward prices, as these will be given in terms linear combination of the factor processes, which are cointegrated with respect to $\mathbb P$.   

\section{Cointegration for Hilbert-valued stochastic processes}

We want to generalize the concept of cointegration to 
Hilbert-valued stochastic processes. Recall from above that we defined cointegration in finite dimensions
by the following scheme of linear mappings for a cointegration pricing system $(\mathcal P,\mathbf c)$ and a factor process 
$\mathbf X\in\R^n$:
$$
\mathbf X\in\R^n\stackrel{\mathcal P}{\longrightarrow} \mathbf S(t)\in\R^d\stackrel{\mathbf c}{\longrightarrow} \mathbf c^{\top}\mathbf S(t)\in\R.
$$
I.e., for a cointegration pricing system $(\mathcal P,\mathbf c)$, we use a linear operator $\mathcal P$ to map the factor vector from the 
{\it factor space} $\R^n$ to the {\it price space} $\R^d$, and next the linear operator $\mathbf c$ to map the price vector from the price space to the real line, which we can think of as the {\it cointegration space}. We lift this to Hilbert-valued
stochastic processes:

Let $\mathsf F$, $\mathsf P$ and $\mathsf C$ be three separable Hilbert spaces, denoting the {\it factor}, 
{\it price} and {\it cointegration} space, resp. We denote $\langle\cdot,\cdot\rangle_i$,  the
inner product with associated norm $|\cdot|_i$, for $i=\mathsf F,\mathsf P,\mathsf C$. Assume $\mathcal P\in L(\mathsf F,\mathsf P)$,
which we call the {\it price operator} and $\mathcal C\in L(\mathsf P,\mathsf C)$ the {\it cointegration operator}. 
For $\{X(t)\}_{t\geq 0}$ being an $\mathsf F$-valued predictable process, we
define the price process 
\begin{equation}
\label{eq:price-hilbert}
Y(t)=\mathcal P X(t)\,, t\geq 0
\end{equation}
which becomes a $\mathsf P$-valued predictable process. 
\begin{definition}
We say that $(\mathcal P,\mathcal C)$ is a cointegration pricing system if the $\mathsf C$-valued stochastic process
$\{\mathcal C\mathcal PX(t)\}_{t\geq0}$ admits a limiting distribution. We say that the price process $Y(t)=\mathcal P X(t)$ 
for given $\mathcal P\in L(\mathsf F,\mathsf P)$ is cointegrated if there exists a $\mathcal C\in L(\mathsf P,\mathsf C)$
such that $\mathcal CY(t)$ admits a limiting distribution.
\end{definition}
Obviously, if $Y$ in \eqref{eq:price-hilbert} is cointegrated, then $(\mathcal C,\mathcal P)$ is a cointegration pricing
system for the given cointegration operator $\mathcal C$. 

We recall from infinite dimensional stochastic analysis (see e.g. Peszat and Zabczyk~\cite{PZ}) that the distribution of a $\mathsf C$-valued random variable $Z$ is defined as the image measure $P_Z$ on the Borel sets $\mathcal B(\mathsf C)$ of $\mathsf C$, that is $P_Z(A)=\mathbb P(Z\in A)$ for $A\in\mathcal B(\mathsf C)$. The definition of cointegration demands the existence of a
probability measure $P_{\infty}$ on $\mathcal B(\mathsf C)$ such that $P_{\mathcal C\mathcal PX(t)}\rightarrow P_{\infty}$ when $t\rightarrow\infty$, where the limit is in the sense of probability measures, e.g., for every bounded measurable function
$g:\mathsf C\rightarrow \R$, it holds for $t\rightarrow\infty$
$$
\int_{\mathsf C}g(u)\,P_{\mathcal C\mathcal PX(t)}(du)\rightarrow \int_{\mathsf C}g(u)P_{\infty}(du).
$$  
Denote the cumulant functional of $X$ by $\Psi_X(t,v)\,, v\in \mathsf F$, defined
as
$$
\Psi_X(t,v):=\log\E\left[\e^{\mathrm{i}\langle v,X(t)\rangle_{\mathsf F}}\right],
$$
where $\log$ is the distinguished logarithm (see  e.g. Sato~\cite{Sato}).
We have the following equivalent characterization of cointegration:
\begin{proposition}
\label{prop:equiv-charact-inf}
$Y$ is cointegrated if and only of there exists a $\mathcal C\in L(\mathsf P,\mathsf C)$
and a cumulant function $\Psi_{\mathcal C}$ such that 
$$
\lim_{t\rightarrow\infty}\Psi_X(t,\mathcal P^*\mathcal C^*u)=\Psi_{\mathcal C}(u)
$$
for all $u\in\mathsf C$.
\end{proposition}
\begin{proof}
It holds that
$\langle u,\mathcal C\mathcal PX(t)\rangle_{\mathsf C}=\langle\mathcal P^*\mathcal C^*u,X(t)\rangle_{\mathsf F}$
for every $u\in\mathsf C$. Thus,
$$
\log \E\left[\e^{\mathrm{i}\langle u,\mathcal C\mathcal P X(t)\rangle_{\mathsf C}}\right]=\Psi_X(t,\mathcal P^*\mathcal C^*u).
$$
If $(\mathcal P,\mathcal C)$ is a cointegrated pricing system, then $\Psi_{\mathcal C}$ is the cumulant function of
$P_{\infty}$. Opposite, the existence of a cumulant function $\Psi_{\mathcal C}$ as the limit of 
$\Psi_X(t,\mathcal P^*\mathcal C^*u)$ yields the existence of a $P_{\infty}$ having $\Psi_{\mathcal C}$ as its 
cumulant. The result follows.
\end{proof}
We next connect cointegration in Hilbert space to cointegration in finite dimensions, as considered in the previous sections:
\begin{proposition}
\label{prop:finite-dim-coint}
Let $(\mathcal C,\mathcal P)$ be a cointegration pricing system (for the factor process $X$ in $\mathsf F$). 
Assume that $\dim(\ker\mathcal C^{\perp}):=d<\infty$ and $\dim(\ker\mathcal P^{\perp}):=n<\infty$ for
$n,d\in\mathbb N$. Then for every $\mathcal T\in\mathsf C^*$, there exist $\mathbf c_{\mathcal T}\in\R^d, \overline{\mathcal P}\in\R^{d\times n}$ and an $\R^n$-valued factor process $\mathbf X(t)$ such that
$$
\mathcal T\mathcal C\mathcal P X(t)=\mathbf c_{\mathcal T}^{\top}\overline{\mathcal P}\mathbf X(t),
$$
and where the real-valued process $t\mapsto\mathbf c_{\mathcal T}^{\top}\overline{\mathcal P}\mathbf X(t)$ admits a limiting
distribution. 
\end{proposition} 
\begin{proof}
For any $u\in\mathsf P$, $u^{\perp}:=u-\text{Proj}_{\ker\,\mathcal C}u\in\ker\mathcal C^{\perp}$ and for an ONB $\{h_i\}_{i=1}^d$ in $\ker\mathcal C^{\perp}$ we find,
$$
u^{\perp}=\sum_{i=1}^d\langle u^{\perp},h_i\rangle_{\mathsf P}h_i=\sum_{i=1}^d\langle u,h_i\rangle_{\mathsf P}h_i
-\sum_{i=1}^d\langle \text{Proj}_{\ker\mathcal C}u,h_i\rangle_{\mathsf P}h_i=\sum_{i=1}^d\langle u,h_i\rangle_{\mathsf P}h_i.
$$
But then, since $\mathcal C u=\mathcal C u^{\perp}$ for every $u\in\mathsf P$, it follows
$$
\mathcal C\mathcal P X(t)=\sum_{i=1}^{d}\langle\mathcal P X(t),h_i\rangle_{\mathsf P}\mathcal C h_i.
$$
Next, for any $v\in\mathsf F$, we have that $v^{\perp}:=v-\text{Proj}_{\ker\mathcal P}v\in\ker\mathcal P^{\perp}$,
and for an ONB $\{f_j\}_{j=1}^n$ in $\ker\mathcal P^{\perp}$ it holds that 
$$
v^{\perp}=\sum_{j=1}^n\langle v,f_j\rangle_{\mathsf F}f_j.
$$
Since $\mathcal P v=\mathcal P v^{\perp}$ for any $v\in\mathsf F$, we derive
$$
\mathcal PX(t)=\sum_{j=1}^n\langle X(t),f_j\rangle_{\mathsf F}\mathcal P f_j.
$$
From this we find 
\begin{equation}
\label{eq:coint-FDR}
\mathcal C\mathcal P X(t)=\sum_{i=1}^d\sum_{j=1}^n\langle X(t),f_j\rangle_{\mathsf F}\langle\mathcal P f_j,h_i\rangle_{\mathsf P}\mathcal C h_i.
\end{equation}
Define the $d\times n$-matrix $\overline{\mathcal P}:=\{\langle\mathcal P f_j,h_i\rangle_{\mathsf P}\}_{i=1,\ldots,d,j=1,\ldots,n}$ and the $\R^n$-valued factor process 
$\mathbf X(t):=(\langle X(t),f_1\rangle_{\mathsf F},\ldots,\langle X(t),f_n\rangle_{\mathsf F})^{\top}$. Finally,
we introduce $\mathbf c_{\mathcal T}:=(\mathcal T\mathcal C h_1,\ldots,\mathcal T\mathcal C h_d)^{\top}\in\R^d$, and
the representation of $\mathcal T\mathcal C\mathcal P X(t)$ follows. 

Note that for any $\theta\in\R$,
$$
\theta\mathcal T\mathcal C\mathcal P X(t)=\langle X(t),\mathcal P^*\mathcal C^*\mathcal T^*\theta\rangle_{\mathsf F}.
$$  
Therefore, by the assumption that $\mathcal C\mathcal P X(t)$ admits a limiting distribution in combination with 
Prop.~\ref{prop:equiv-charact-inf}, there exists a
function $\R\ni\theta\mapsto \Psi_{\mathcal T\mathcal C}(\theta)\in\mathbb C$ given by
$$
\Psi_{\mathcal T\mathcal C}(\theta)=\lim_{t\rightarrow\infty}\Psi_X(t,\mathcal P^*\mathcal C^*\mathcal T^*\theta)
=\Psi_{\mathcal C}(\mathcal T^*\theta).
$$
The function $\Psi_{\mathcal T\mathcal C}$ is a cumulant function, since $\mathcal T$ is a continuous linear operator
(see e.g. Sato~\cite[Prop.~2.5 (viii)]{Sato}). 
The Proposition follows. 
\end{proof}
Notice that for any $\mathcal T\in\mathsf C^*$, $\mathcal T\mathcal C\in\mathsf P^*$, and we can interpret
$(\mathcal T\mathcal C,\mathcal P)$ as a cointegration pricing system with cointegration space $\R$. This holds
for general cointegration pricing system $(\mathcal C,\mathcal P)$ and not only those for which the 
complement space of the kernels of $\mathcal C$ and $\mathcal P$ are finite.

We see from the proof of Proposition~\ref{prop:finite-dim-coint} that only $\mathbf c_{\mathcal T}$ is depending on 
$\mathcal T$, which explains the subscript. Given the factor process $\mathbf X$ in Prop.~\ref{prop:finite-dim-coint},
it follows that $(\mathbf c_{\mathcal T},\overline{\mathcal P})$ is a cointegration pricing system, and indeed,
$\mathbf c_{\mathcal T}$ is a cointegration vector for the price vector $\mathbf S(t)=\overline{\mathcal P}\mathbf X(t)$.
The vector $\mathbf c_{\mathcal T}$ is further depending on $\mathcal C$, naturally. The pricing matrix
$\overline{\mathcal P}$ depends on the basis functions $\{f_j\}_{j=1}^n$ in $\ker\mathcal P^{\perp}$
and $\{h_i\}_{i=1}^d$ in $\ker\mathcal C^{\perp}$. Thus, it depends on the cointegration
pricing system $(\mathcal C,\mathcal P)$. The finite-dimensional factor process $\mathbf X$ depends on the 
basis $\{f_j\}_{j=1}^n$, and thus on the pricing operator $\mathcal P$. 

From \eqref{eq:coint-FDR} it follows that $\mathcal C\mathcal P X(t)$ is a process with values in the
finite-dimensional subspace $\text{span}\{\mathcal C h_1,\ldots,\mathcal C h_d\}$ of $\mathsf C$ when $\ker\mathcal C^{\perp}$
and $\ker\mathcal P^{\perp}$ have finite dimension. Thus, we may introduce the following definition:
\begin{definition}
\label{def:FDR}
A cointegration pricing system $(\mathcal C,\mathcal P)$ has a finite dimensional realization (FDR) if, for $n,d\in\mathbb N$, there exist an $\R^n$-valued factor process $\mathbf X$, a $d\times n$ pricing matrix $\overline{\mathcal P}$ and
a $c\in \mathsf C^{\times d}$ such that $\mathcal C\mathcal P X(t)=c^{\top}\overline{\mathcal P}\mathbf X(t)$.   
\end{definition}
In view of Proposition~\ref{prop:finite-dim-coint}, we have an FDR when $\ker\mathcal C^{\perp}$ and $\ker\mathcal P^{\perp}$ are
finite dimensional. In this case, $c=(\mathcal C h_1,\ldots,\mathcal C h_d)^{\top}$ with $\{h_i\}_{i=1}^d$
being the ONB of $\ker\mathcal C^{\top}$. If $(\mathcal C,\mathcal P)$ is a general cointegration pricing
system which has an FDR, then for any $\mathcal T\in\mathsf C$ we find that 
$$
\mathcal T\mathcal C\mathcal P X(t)=\mathbf c_{\mathcal T}^{\top}\overline{\mathcal P}\mathbf X(t),
$$
for $\mathbf c_{\mathcal T}:=\mathcal Tc=(\mathcal Tc_1,\ldots\mathcal Tc_d)^{\top}\in\R^d$. Hence, 
$(\mathcal Tc,\overline{\mathcal P})$ will be a finite dimensional cointegrated pricing system for the
factor process $\mathbf X$.
We remark that Definition~\ref{def:FDR} does not really rest on the fact that there exist any limiting distribution, 
but as we work
with cointegration in this paper, we focus on cointegration pricing systems, that is, pricing systems $(\mathcal C,\mathcal P)$
for which $\mathcal C\mathcal P X(t)$ admits a limiting distribution. 

We remark that we have not assumed any minimality of the pricing matrix $\overline{\mathcal P}$ in the above 
considerations. We recall from the proof of Prop.~\ref{prop:finite-dim-coint} that the $d\times n$-matrix
$\overline{\mathcal P}$ has elements $\langle\mathcal P f_j,h_i\rangle_{\mathsf P}$, and minimality is achieved as
long as this matrix has full rank. However, the next proposition shows that we must take into account a possible
finite-dimensionality of the factor process $X$ as well.
Indeed, another situation where we may have an FDR is when the factor process has a finite-dimensional state space:
\begin{proposition}
Assume that the factor process $\{X(t)\}_{t\geq 0}$ takes values in $\mathsf F_n\subset\mathsf F$, where
$\dim(\mathsf F_n):=n<\infty$ for $n\in\mathbb N$. Then any cointegration pricing system
$(\mathcal C,\mathcal P)$ has a finite dimensional realization, with $\overline{\mathcal P}=\text{Id}$ (the identity
matrix on $\R^n$),  $\mathbf X(t):=(\langle X(t),f_1\rangle_{\mathsf F},\ldots,\langle X(t),f_n\rangle_{\mathsf F})^{\top}\in\R^n$
for an ONB $\{f_j\}_{j=1}^n$ of $\mathsf F_n$ and $c=(\mathcal C\mathcal Pf_1,\ldots,\mathcal C\mathcal P f_n)^{\top}
\in\mathsf C^{\times n}$.  
\end{proposition}
\begin{proof}
If $X(t)\in\mathsf F_n$, then $X(t)=\sum_{j=1}^n\langle X(t),f_j\rangle_{\mathsf F}f_j$ and therefore
$$
\mathcal C\mathcal P X(t)=\sum_{j=1}^n\langle X(t),f_j\rangle_{\mathsf F}\mathcal C\mathcal P f_j=c^{\top}
\text{Id}\mathbf X(t).
$$
The result follows.
\end{proof}
This result indicates strongly the possible non-uniqueness of the FDR, since depending on the pricing operator
$\mathcal P$, one may specify a different $\overline{\mathcal P}$ than the identity matrix, and thus also different 
$c$. It also shows that the question of minimality depends on $\mathcal C, \mathcal P$ and the possible finite
dimensionality of $X$. 

Let us now focus on the case where $\mathsf F$ and $\mathsf P$ can be represented
as product spaces, e.g., when $\mathsf F=\mathsf H^{\times n}$ and $\mathsf P=\mathsf K^{\times d}$
for two separable Hilbert spaces $\mathsf H$ and $\mathsf K$. We denote the inner product as usual by 
$\langle\cdot,\cdot\rangle_{i}$ with corresponding norms $|\cdot|_{i}$, where the subscript indicates the space, here $i=\mathsf H,\mathsf K$. The inner product on the product space $\mathsf F$ is then given by
$\langle u,v\rangle_{\mathsf F}=\sum_{j=1}^n\langle u_j,v_j\rangle_{\mathsf H}$ for 
$u=(u_1,\ldots,u_n)\in\mathsf F$ and $v=(v_1,\ldots,v_n)\in\mathsf F$ (and likewise for $\mathsf P$). 

To make an example, suppose we have given a factor process $X\in\mathsf H^{\times n}$ and a pricing operator $\mathcal P$ 
given as an $d\times n$-matrix of operators $\mathcal P=\{\mathcal P_{ij}\}_{i=1,\ldots,d,j=1,\ldots,n}$ with
$\mathcal P_{ij}\in L(\mathsf H,\mathsf K)$. Then the pricing vector will be $Y(t)=\mathcal P X(t)$, which is a
$\mathsf K^{\times d}$-valued stochastic process. Indeed, we have that $Y=(Y_1,\ldots,Y_d)^{\top}$ with
$$
Y_i(t)=\sum_{j=1}^n\mathcal P_{ij}X_j(t)
$$
for $i=1,\ldots,d$. In analogy with Example~\ref{motivating-ex}, we assume that $(X_1,\ldots,X_{n-1})^{\top}\in\mathsf H^{\times (n-1)}$ admits a limiting distribution, while $X_n$ may be non-stationary. We observe that any $\mathcal C=(\mathcal C_1,\ldots,\mathcal C_d)^{\top}$ with $\mathcal C_i\in L(\mathsf K,\mathsf C)$ will be such that $\mathcal C\in L(\mathsf K^{\times d},\mathsf C)$. Under the condition
$\sum_{i=1}^d\mathcal C_i\mathcal P_{in}=0$ we find $\mathcal C Y(t)=\sum_{i=1}^d\sum_{j=1}^{n-1}\mathcal C_i\mathcal P_{ij}X_j(t)$, that is, a $\mathsf C$-valued stochastic process not depending on $X_n$ but only on $X_j$ for $j=1,\ldots,n-1$. This
provides us with a simple example of a cointegration pricing system. 

A way to generate a system of factor processes $X\in\mathsf H^{\times n}$ can be as follows: consider an $\R^m$-valued stochastic process $\{\mathbf Z(t)\}_{t\geq 0}$ and
$\mathcal A\in L(\R^m,\mathsf H^{\times n})$. For $b\in\mathsf H^{\times n}$, define the factor process
$$
X(t)=\mathcal A\mathbf Z(t)+b.
$$
We remark that $\mathcal A$ can be represented as an $n\times m$-matrix with elements in $\mathsf H$. Indeed, the 
columns of this matrix will be given by the action of $\mathcal A$ on the canonical basis vectors in $\R^m$. If $\mathsf H$
is some space of functions on $\R_+$, we may relate the factor process $X$ to the affine models of forward prices from the previous section, i.e., the affine
forward models provide a class of factors in an infinite dimensional framework. The existence of a limiting distribution of one or more of the
factors $X_j, j=1,\ldots,n$ can be traced back to the process $\mathbf Z$. Indeed, this simplified case relates us back to the models
consider in Section~\ref{sect:forward}, for example the polynomial processes in Proposition~\ref{prop:polynomial}.

\subsection{A discussion of cross-commodity forward markets}

Let us now focus specifically on commodity forward markets, and start with a discussion on cross-commodity models. 
Suppose we have $d$ forward markets, with forward price dynamics denoted by $f_i(t,x)$, $i=1,\ldots,d$ and 
$x\in\R_+$ being time to maturity. We are aiming at a $d$-dimensional model of the forward
curve dynamics $t\mapsto f(t,\cdot)=(f_1(t,\cdot),\ldots,f_d(t,\cdot))^{\top}$. We choose $\mathsf H$ to be a Hilbert space
of real-valued measurable functions on $\R_+$. Following the analysis in Benth and Kr\"uhner~\cite{BK-sifin}, a 
convenient choice of such a space could be the so-called Filipovic space of absolutely continuous functions (see Appendix~\ref{app:filipovic}
for a definition). 

Based on the analysis in Benth and Kr\"uhner~\cite{BK-sifin} (see also Benth and Kr\"uhner~\cite{BK-CIMS}), the forward price dynamics $\{f(t,\cdot)\}_{t\geq 0}$
can be expressed as a $\mathsf H^{\times d}$-valued stochastic process 
\begin{equation}
\label{eq:cross-comm-forward-model}
df(t,\cdot)=\partial f(t,\cdot)\,dt+\beta(t,f(t,\cdot))\,dt+\sigma(t,f(t,\cdot))\,dL(t)
\end{equation}
where $L$ is a $\mathsf V$-valued square-integrable L\'evy process with zero mean and $\mathsf V$ being a separable Hilbert space. 
We use the notation $\partial$ for the $d\times d$ matrix-operator
\begin{equation}
\partial=\left[\begin{array}{cccc}  \frac{\partial}{\partial x} & 0 & \cdots & 0 \\
0 & \frac{\partial}{\partial x} & \cdots & 0 \\
..  &  . & \cdots & . \\
0 & 0 & \cdots & \frac{\partial}{\partial x}
\end{array}\right],   
\end{equation}
with $\partial/\partial x$ being the derivative operator on the functions in $\mathsf H$. We assume that this operator is 
a densely defined unbounded operator on $\mathsf H$ which is the generator
of a $C_0$-semigroup (the shift semigroup). This holds if we choose $\mathsf H$ to be the Filipovic space, say. 
Further, the measurable mappings $\sigma:\R_+\times \mathsf H^{\times d}\rightarrow L(\mathsf V,\mathsf H^{\times d})$ and 
$\beta:\R_+\times \mathsf H^{\times d}\rightarrow\mathsf H^{\times d}$ are assumed to satisfy the Lipschitz conditions stated in
Peszat and Zabczyk~\cite[Section~9.2]{PZ} such that there exists a unique mild predictable cadlag solution 
to \eqref{eq:cross-comm-forward-model}.

The function $\beta$ models the risk premium in this cross-commodity model of forward curves. We note in passing that 
\eqref{eq:cross-comm-forward-model} is formulated under $\mathbb P$, and to ensure an arbitrage-free dynamics 
there must exist a probability $\mathbb Q\sim\mathbb P$ such that the $\mathbb Q$-dynamics of $f$ is
$$
df(t,\cdot)=\partial f(t,\cdot)\,dt+dM(t)
$$
where $M$ is a $\mathsf H^{\times d}$-valued  (local) $\mathbb Q$-martingale (see Benth \& Kr\"uhner~\cite{BK-CIMS}). We will not pursue the existence of
such a $\mathbb Q$ in further detail here.

We may view the cross-commodity forward model \eqref{eq:cross-comm-forward-model} in our contegration context by choosing the
factor process $X$ to be equal to the price vector process $f$. Thus, we have $\mathsf H=\mathsf K$ and $n=d$, with a pricing
matrix $\mathcal P$ simply being the identity operator in $\mathsf H^{\times d}$. In particular, we let $\mathsf P=\mathsf H^{\times d}$,
i.e., the pricing space is the product space.
In many markets, prices are naturally varying over seasons. For example in power markets, prices are typically higher in heating and cooling
seasons. Such a behaviour may be modelled into $\beta$. Further, many commodities are based on extinguishable resources, with oil and gas as prime examples. For such commodities, one may expect non-stationarity effects in prices.  Other sources of non-stationarity are technological changes and inflation. Such non-stationarity could possibly be modelled in
the $\beta$, as well, for example by adding dependency on additional (non-stationary) stochastic factors $Z$, e.g., 
assuming a drift of the form $\beta(t,f(t,\cdot),Z(t))$. The additional factors $Z$ may be Hilbert-valued processes. 

Cointegration in this context could be formulated as follows: There is an operator 
$\mathcal C\in L(\mathsf H^{\times d},\mathsf H)$ such that 
the $\mathsf H$-valued stochastic process $t\mapsto g(t,\cdot):=\mathcal C f(t,\cdot)$ admits a limiting distribution. In many applications one is interested in the spread between two or more forward markets, and it is natural to consider linear combinations of the forward curves, which again will be an element in the space of (marginal) forward curves. This gives a rationale for choosing $\mathsf C=\mathsf H$. If we assume the rather strong condition that $\partial$ commutes with
$\mathcal C$ in the sense that $\mathcal C\partial=\frac{\partial}{\partial x}\mathcal C$ on $\text{Dom}(\partial)$, we find 
the stochastic dynamics of $g$ to be 
$$
dg(t,\cdot)=\frac{\partial}{\partial x}g(t,\cdot)\,dt+\mathcal C\beta(t,f(t,\cdot),Z(t))\,dt+\mathcal C\sigma(t,f(t,\cdot))\,dL(t).
$$  
Thus, 
$$
g(t,\cdot)=\mathcal S(t)g_0(\cdot)+\int_0^t\mathcal S(t-s)\mathcal C\beta(s,f(s,\cdot),Z(s))\,ds+\int_0^t\mathcal S(t-s)\mathcal C\sigma(s,f(s,\cdot))\,dL(s),
$$
where $\mathcal S$ is the $C_0$-semigroup generated by $\partial/\partial x$ on $\mathsf H$ (the shift semigroup, also called
the translation semigroup), and $g(0,\cdot)=\mathcal C f(0,\cdot)=:g_0(\cdot)\in\mathsf H$. 
The existence of a limiting distribution is closely linked to properties of the $C_0$-semigroup along with $\beta$ and $\sigma$. 
Specializing to $L=W$, a Wiener process, and $\mathsf H$ being the Filipovic space, we may resort to 
Tehranchi~\cite{Tehranchi} for sufficient conditions for the existence of an invariant measure of $g$. In particular, these 
conditions will include the time-homogeneity and Lipschitzianity of $\beta$ and $\sigma$. We remark in passing that Tehranchi~\cite{Tehranchi} treats HJM models, which has a nonlinearity in the drift satisfying a no-arbitrage condition with the volatility $\sigma$. In our context we will
have a simplified situation where this drift condition is not needed.

As a specific case, we could consider the highly dependent power forward markets in Germany and France. In Germay, there has been a gradual increase of renewable power generation from photovoltaic and wind, and we
let $Z(t)$ be a real-valued stochastic process measuring the total generation of such. Since the amount of sunshine over the 
day is varying with
season, and so is the average wind speed, one has that $Z$ is likely to vary seasonally. Moreover, with the "Energiewende" still 
in place, the process will likely show an increasing trend, at least on a short term horizon. Hence, $Z$ may be thought of as a non-stationary 
stochastic process. Assume now that $\beta(t,f(t,.),Z(t))=(\beta_1Z(t),\beta_2Z(t))^{\top}$, for $\beta_1,\beta_2$ two constants, 
which is an $\R^2$-valued stochastic process,
and thus trivially in $\mathsf H^{\times 2}$. Further, we let the volatility be constant, in the sense that 
$\sigma(s,f(s,\cdot))=\Sigma\in L(\mathsf V,\mathsf H^{\times 2})$. Under this specification, we choose $\mathcal C^{\top}:=(\beta_2,-\beta_1)$, which will commute with $\partial/\partial x$, and we find for 
$g(t,\cdot):=\beta_2f_1(t,\cdot)-\beta_1f_2(t,\cdot)$  
\begin{equation}
\label{g-eq}
g(t,\cdot)=\mathcal S(t)g_0(\cdot)+\int_0^t\mathcal S(t-s)\mathcal C^{\top}\Sigma\,dL(s).
\end{equation}
The cointegration process $g$ will be an Ornstein-Uhlenbeck process with unbounded operator $\partial/\partial x$ and
volatility $\mathcal C^{\top}\Sigma$. Invariant measures for L\'evy-driven Ornstein-Uhlenbeck processes are thoroughly discussed in Applebaum~\cite{Apple-review} (see also references therein). Although Tehranchi~\cite{Tehranchi} considers more general HJM-models 
with Gaussian noise, one can apply his methods to conclude that $g$ in \eqref{g-eq} admits a limiting distribution if
we choose $\mathsf H$ to be the Filipovic space (see Appendix~\ref{app:filipovic}). We remark in passing that Tehranchi~\cite{Tehranchi} makes use of the 
fact that the shift semigroup $\mathcal S(t)$ is a strict contraction on a convenient subspace of the Filipovic space.


So far we have only considered arithmetic forward models. To introduce a geometric model, of the form $F(t,x):=\exp(f(t,x))$, with
$f$ defined by the dynamics \eqref{eq:cross-comm-forward-model} and $\exp(f):=(\exp(f_1),\ldots,\exp(f_d))$, we must impose additional structure on the Hilbert space $\mathsf H$. Indeed, it has to be closed under
exponentiating, that is, for any $h\in\mathsf H$, it must hold that $\exp h\in\mathsf H$. If $\mathsf H$ is a Banach algebra under pointwise multiplication, this holds true, since in that case we have $|h^n|_{\mathsf H}\leq|h|_{\mathsf H}^n$ and thus $|\exp h|_{\mathsf H}\leq 
\exp |h|_{\mathsf H}<\infty$. We remark that after an appropriate scaling of the norm in the Filipovic space, it becomes a Banach algebra
(see Benth and Kr\"uhner~\cite{BK-CIMS}). 

\subsection{A three-factor example}
We end this Section with a concrete example adopted from Benth~\cite{benth-eberleinfest}. Let $\mathsf H$ be a Hilbert space of real-valued measurable functions on $\R_+$. Consider a three factor processes $X=(X_1,X_2,X_3)^{\top}\in\mathsf H^{\times 3}$ given by $X_3(t)=L(t)$ where $L$ is an $\R$-valued L\'evy process and for $x\in\R_+$,
\begin{equation}
\label{eq:eberlein-model}
X_k(t,x)=h_k(t,x)+\int_0^tg_k(t+x-s)\,dU_k(s), k=1,2.
\end{equation}
Here, for $k=1,2$, $U_k$ are $\R$-valued L\'evy processes
with zero mean and finite variance, and 
$h_k(t,\cdot),g_k\in\mathsf H$.  In the next lemma, we state conditions
such that $\{X_k(t)\}_{t\geq 0}$ becomes an $\mathsf H$-valued stochastic process.
\begin{lemma}
\label{lemma:lss}
Suppose that the shift semigroup $\{\mathcal S(t)\}_{t\geq 0}$ is bounded on $\mathsf H$, i.e., $\mathcal S(t)\in L(\mathsf H)$ for all $t\geq 0$. 
If $\int_0^t|g_k(s+\cdot)|_{\mathsf H}^2\,ds<\infty$ for every $t\geq 0$, then $\{X_k(t)\}_{t\geq 0}$ defined in \eqref{eq:eberlein-model} is an $\mathsf H$-valued stochastic process. Its cumulant is
$$
\log\E\left[\exp\left(\mathrm{i}(h,X_k(t))_{\mathsf H}\right)\right]=\mathrm{i}(h,h_k(t))_{\mathsf H}+\int_0^t\psi_{U_k}\left(
(h,g_k(s+\cdot))_{\mathsf H}\right)\,ds,
$$
for $h\in\mathsf H$ and $\psi_{U_k}$ the cumulant of $U_k(1)$. 
\end{lemma} 
\begin{proof}
Fix $t\geq 0$. By assumption, it holds that 
$g_k(t-s+\cdot)=\mathcal S(t-s)g_k(\cdot)\in\mathsf H$ for all $s\in[0,t]$. From Peszat and Zabczyk~\cite{PZ}, 
the stochastic integral $\int_0^tg_k(t-s+\cdot)\,dU_k(s)$ is well-defined and defines an element in $\mathsf H$ if
$\int_0^t|g_k(t-s+\cdot)|_{\mathsf H}^2\,ds<\infty$, which holds by assumption. Thus,  $\{X_k(t)\}_{t\geq 0}$ is
an $\mathsf H$-valued stochastic process. 

We have that the operator $G(t-s)(h)=(h,g_k(t-s+\cdot))_{\mathsf H}$ is a linear functional on $\mathsf H$. Moreover, by the Cauchy-Schwartz inequality,
\begin{align*}
\int_0^tG^2(t-s)(h)\,ds&=\int_0^t(h,g_k(t-s+\cdot))_{\mathsf H}^2\,ds \\
&\leq |h|_{\mathsf H}^2\int_0^t|g_k(s+\cdot)|^2_{\mathsf H}\,ds.
\end{align*} 
Hence, by the integrability assumption on the norm of $g_k$, $s\rightarrow G(t-s)(h)$ is $U_k$-integrable on $[0,t]$, and by linearity 
we find 
$$
(h,\int_0^tg_k(t-s+\cdot)\,dU_k(s))_{\mathsf H}=\int_0^t(h,g_k(t-s+\cdot))_{\mathsf H}\,dU_k(s)=\int_0^tG_k(t-s)(h)\,dU_k(s).
$$
Hence,
\begin{align*}
\log \E\left[\exp\left(\mathrm{i}(h,\int_0^tg_k(t-s+\cdot)\,dU_k(s))_{\mathsf H}\right)\right]&=\log\E\left[\mathrm{i}
\int_0^tG(t-s)(h)\,dU_k(s)\right] \\
&=\int_0^t\psi_{U_k}\left(G(s)(h)\right)\,ds.
\end{align*}
Since $U_k$ is a zero mean square integrable L\'evy process, its cumulant becomes 
$$
\psi_{U_k}(z)=-\frac12\sigma^2z^2+\int_{\R}(\e^{\mathrm i z y}-1-\mathrm{i}zy)\,\ell(dy)
$$
for $\sigma\geq 0$ a constant and $\ell$ the L\'evy measure (see Applebaum~\cite{apple-book}). We have 
$$
|\e^{\mathrm i z y}-1-\mathrm{i}zy|=|(\mathrm{i}z)^2\int_0^y\int_0^x\e^{\mathrm{i}zu}\,du\,dx|\leq \frac12z^2y^2,
$$
and therefore $\psi_{U_k}(G(s)(h))$ is integrable on $[0,t]$ whenever $G(s)(h)\in L^2([0,t])$, which holds by assumption
after appealing to the Cauchy-Schwartz inequality, as argued above.  
\end{proof}
In the Lemma~\ref{lemma:lss} above we assumed that the function $[0,t]\ni s\mapsto |g_k(s+\cdot)|_{\mathsf H}\in \R_+$ is in $L^2([0,t])$. As the shift operator $\mathcal S(t)$ is assumed continuous, a sufficient condition for this to hold is that $s\mapsto\|\mathcal S(s)\|_{\text{op}}\in L^2([0,t])$. Whenever the family
of shift operators defines a strongly continuous semigroup, say, this holds true. If in addition $\{\mathcal S(t)\}_{t\geq 0}$ is
exponentially stable, we have that $s\mapsto\|\mathcal S(s)\|_{\text{op}}\in L^2(\R_+)$. If 
$s\mapsto |g_k(s+\cdot)|_{\mathsf H}\in L^2(\R_+)$
and $h_k(t)$ has a limit in $\mathsf H$ as $t\rightarrow\infty$, it follows from Lemma~\ref{lemma:lss} that $\{X_k\}_{t\geq 0}$ admits a 
limiting distribution in $\mathsf H$. 

Introduce next the pricing operator $\mathcal P\in L(\mathsf H^{\times 3},\mathsf H^{\times 2})$ simply as
\begin{equation}
\mathcal P=\left[\begin{array}{ccc} \text{Id} & 0 & \text{Id} \\ 0 & \text{Id} & \text{Id} \end{array}\right],
\end{equation}
where $\text{Id}$ is the identity operator on $\mathsf H$. If we assume $\mathsf H$ to be a Banach algebra, we can define the
exponential forward price dynamics for a bivariate commodity market by
\begin{equation}
F(t):=\exp\left(\mathcal P X(t)\right).
\end{equation}
Following the analysis in Benth~\cite{benth-eberleinfest}, we can choose $h_k(t)$ to ensure an arbitrage-free dynamics (see Prop.~2 in
\cite{benth-eberleinfest}). One can also think of $h_k$ as a model for the market price of risk/risk premium in the forward market.

In this bivariate cross commodity forward price model, we see that $\ln F(t)=\mathcal P X(t)$, and thus for any $\mathcal C\in L(\mathsf H^{\times 2},\mathsf C)$, we have 
$$
\mathcal C\ln F(t)=\mathcal C_1X_1(t)+\mathcal C_2 X_2(t)+(\mathcal C_1+\mathcal C_2)X_3(t).
$$
Here we have represented the operator $\mathcal C$ in matrix form, i.e.,
$$
\mathcal C=\left[\begin{array}{cc} \mathcal C_1 & \mathcal C_2\end{array}\right]
$$
for $\mathcal C_i\in L(\mathsf H,\mathsf C), i=1,2$. Letting $\mathcal C_2=-\mathcal C_1$, we find $\mathcal C\ln F(t)=\mathcal C_1(X_1(t)-X_2(t))$.  In the next lemma, we state sufficient conditions for $\mathcal C_1(X_1(t)-X_2(t))$ to admit a limiting distribution in
$\mathsf H$, which thus yield sufficient conditions for having a cointegrated model. 
\begin{lemma}
\label{lemma:limit-lss}
Assume that the shift operator $\mathcal S(t)$ is bounded in $\mathsf H$ for all $t\geq 0$ and $|g_k(s+\cdot)|_{\mathsf H}\in L^2(\R_+)$ for $k=1,2$. If $h_\infty := \lim_{t\rightarrow\infty}(h_1(t)-h_2(t))$ exists in 
$\mathsf H$, then $\mathcal C_1(X_1(t)-X_2(t))$ admits a limiting distribution in $\mathsf C$. This 
limiting distribution has cumulant
\begin{align*}
&\lim_{t\rightarrow\infty}\log\E\left[\exp\left(\mathrm{i}(h,\mathcal C_1(X_1(t)-X_2(t)))_{\mathsf C}\right)\right] \\
&\quad
=\mathcal C_1h_\infty+\int_0^{\infty}\psi_U\left((\mathcal C_1^*h,g_1(s+\cdot))_{\mathsf C},-(\mathcal C_1^*h,g_2(s+\cdot))_{\mathsf C}\right)\,ds
\end{align*}
where $\psi_U$ is the cumulant of the bivariate L\'evy process $U=(U_1,U_2)$ and $h\in\mathsf C$. 
\end{lemma}
\begin{proof}
We find, following Lemma~\ref{lemma:lss}, that the processes $\int_0^tg_k(t-s+\cdot)\,dU_k(s)$ in $\mathsf H$ both admit a limiting
distribution. Moreover, by using the same argument for marginal integrability as in the proof of Lemma~\ref{lemma:lss}, we
find that $s\mapsto \psi_U((\mathcal C_1^*h,g_1(s+\cdot))_{\mathsf H},-(\mathcal C_1^*h,g_2(s+\cdot))_{\mathsf H})$ 
is integrable on $\R_+$ for any $h\in\mathsf C$. The result follows.  
\end{proof}
A special case is to choose $\mathsf C=\mathsf H$ and $\mathcal C_1=\text{Id}$. Thus, $X_1(t)-X_2(t)$ admits in particular
a limiting distribution when the conditons in Lemma~\ref{lemma:limit-lss} are fullfilled.   
 
Geman and Liu~\cite{GemanLiu} perform an empirical analysis of cointegration between the gas forward markets at Henry Hub (US) and 
National Balancing Point (UK). They introduce various measures on the forward curves to study how integrated the markets are. 
More specifically, it is proposed to measure the distance between the average of the respective forward curves, or simply the distance between the implied spot prices (closest maturity forwards), or the distance between some geometric weighted average of forward prices. In our 
context, the latter two distance measures can be expressed as $|\mathcal C_1 X_1(t)-\mathcal C_1 X_2(t)|$ with $\mathsf C=\R$ and
$\mathcal C_1\in\mathsf H^*$. For example, in the case of closest forwards (or spot), we choose $\mathcal C_1=\delta_0$, the evaluation 
operator at zero, assuming that this is continuous on $\mathsf H$. A weighted geometric average of the curve, on the other hand, can be translated into a weighted sum of log-prices over different maturities, which gives rise to a linear operator $\mathcal C_1$ being a 
weighted sum of evaluation maps $\delta_x$ for different $x$. The average of the forward curve is not possible to represent via 
a linear operator $\mathcal C_1$ in a geometric model. However, if we choose to work with an arithmetic model, this would simply 
become an
integral operator on the curves in $\mathsf H$.  

In view of the results in Section~\ref{sect:forward}, one can find spot models that leads to cointegration of forward prices with given time
to maturity. In the context of Geman and Liu~\cite{GemanLiu}, measuring the difference of the average of the 
forward curves at given maturity-times could lead to stationarity and thus the conclusion that the markets are cointegrated. 
However, Geman and Liu~\cite{GemanLiu} do not find evidence for cointegration of the two gas forward markets in
Henry Hub and National Balancing Point. This could be explained by a possible term structure of the risk premium (which can be traced back in the $\beta$ function above) and thus the need for more sophisticated choices of operators $\mathcal C$ to reveal a potential
cointegration.

\appendix

\section{The Filipovic space}
\label{app:filipovic}

We present the Filipovic space following Filipovic~\cite{Filip}:
Let $w:\R_+\rightarrow\R_+$ be a monotonely increasing function with $w(0)=1$ and $\int_0^{\infty}w^{-1}(x)\,dx<\infty$. 
Introduce the Filipovic space,
denoted $\mathsf H_w$, as the space of absolutely continuous functions $f:\R_+\rightarrow\R$ for which
$$
|f|_w^2:=f^2(0)+\int_0^{\infty}w(x)(f'(x))^2\,dx<\infty,
$$
where $f'$ is the weak derivative of $f$. With the inner product 
$$
(f,g)_w=f(0)g(0)+\int_0^{\infty}w(x)f'(x)g'(x)\,dx
$$
for $f,g\in \mathsf H_w$, $\mathsf H_w$ becomes a separable Hilbert space. The shift operator $\mathcal S(t):f\mapsto f(t+\cdot)$
for $t\geq 0$ defines a $C_0$-semigroup on $\mathsf H_w$ which is quasi-contractive and uniformly bounded. The generator of $\mathcal S(t)$ is the derivative operator. The evaluation
map $\delta_x:f\mapsto f(x)$ is a linear functional on $\mathsf H_w$. Finally, from Benth and Kr\"uhner~\cite{BK-CIMS}, 
$\mathsf H_w$ becomes a Banach algebra after appropriate rescaling of the norm $|\cdot|_w$, that is, if $f,g\in H_w$, then
$fg\in H_w$ and $\|fg\|_w\leq \|f\|_w\|g\|_w$ with $\|\cdot\|_w:=c|\cdot|_w$ for a suitable constant $c>0$ depending on 
$\int_0^{\infty}w^{-1}(x)\,dx$.

\end{document}